\documentclass{article}
\usepackage{amsmath, amsthm, amssymb, mathrsfs}
\allowdisplaybreaks[1]
\theoremstyle{plain}
\newtheorem{thm}{Theorem}
\newtheorem{prop}{Proposition}
\newtheorem{cor}{Corollary}
\newtheorem{lem}{Lemma}
\newtheorem*{assumption}{Assumption}

\theoremstyle{definition}
\newtheorem{rem}{Remark}

\newcommand{\Avg}{\mathrm{Avg}}

\title{On the value-distribution of the logarithms of symmetric square $L$-functions 
in the level aspect}
\author{
Philippe Lebacque,
Kohji Matsumoto
and 
Yumiko Umegaki
}
\date{}
\begin{document}
\maketitle
%
%
\begin{abstract}
We consider the value distribution of logarithms of symmetric square
$L$-functions associated with newforms of even weight and prime power level
at $s=\sigma > 1/2$.
We prove that certain averages of those values can be written as integrals 
involving a density function which is related with the Sato-Tate measure.
Moreover, we discuss the case of symmetric power $L$-functions.
\end{abstract}
%
%
\section{Introduction and the statement of main results}
\par
The prototype of the theory of $M$-functions is the limit theorem of Bohr and Jessen
\cite{bj} for the Riemann zeta-function $\zeta(s)$, where $s=\sigma+i\tau\in\mathbb{C}$.
Let $R$ be a rectangle in the complex plane with the edges parallel to the axes.
Let $\sigma>1/2$, and for any $T>0$, let $L_{\sigma}(T,R)$ be the 1-dimensional
Lebesgue measure of the set $\{\tau\in[0,T]\mid \log\zeta(\sigma+i\tau)\in R\}$.
Then the Bohr-Jesen theorem asserts the existence of a continuous non-negative function
$\mathcal{M}_{\sigma}(w,\zeta)$ defined on $\mathbb{C}$, for which
\begin{align}\label{BJ-thm}
\lim_{T\to\infty}\frac{L_{\sigma}(T,R)}{T}=\int_R \mathcal{M}_{\sigma}(w,\zeta)
\frac{dudv}{2\pi}
\end{align}
holds with $w=u+iv\in\mathbb{C}$.   Here, the left-hand side is the average of
$\log\zeta(s)$ in the $\tau$-aspect.

Analogous results for Dirichlet $L$-functions $L(s,\chi)$ (and more generally,
$L$-functions over number fields and function fields)
were discussed much later by Ihara
\cite{ihara} and a series of joint papers by Ihara and the second-named author
\cite{im_2010} \cite{im_2011} \cite{im-moscow} \cite{im_2014}.
Their results assert the existence of a continuous non-negative function
$\mathcal{M}_{\sigma}(w,L)$ for which formulas of the form
\begin{align}\label{IM-thm}
{\rm Avg}_{\chi}\Phi(\log L(s,\chi))=\int_{\mathbb{C}}\mathcal{M}_{\sigma}(w,L)\Phi(w)
\frac{dudv}{2\pi}
\end{align}
are valid, where $\Phi$ is a certain test function and ${\rm Avg}_{\chi}$ means a certain
average with respect to $\chi$.
The functions $M_{\sigma}$ appearing on the right-hand sides of \eqref{BJ-thm} and
\eqref{IM-thm} are now called the $M$-functions associated with $\log\zeta(s)$ and
with $\log L(s,\chi)$, respectively.

After those studies, several mathematicians tried to obtain the same type of formulas
for other zeta and $L$-functions.   In particular, $M$-functions associated with
automorphic $L$-functions have been studied in recent years.

This direction of research was first cultivated by two papers published in 2018.
The level and the modulus aspects of automorphic $L$-functions were discussed by
\cite{lz}, while the difference of logarithms of two symmetric power $L$-functions
were treated in \cite{mu}.

As for the $\tau$-aspect, formulas analogous to \eqref{BJ-thm} were obtained 
in \cite{mu-JNT} for automorphic $L$-functions, and in \cite{mu-ropar} for symmetric
power $L$-functions.   Mine's article \cite{mineLMJ} is also to be mentioned,
in which a more general framework is treated.

Recently, Mine \cite{mine} obtained a quite satisfactory analogue of \eqref{IM-thm} 
for automorphic $L$-functions in the level aspect.

The aim of the present paper is to obtain an analogue of \eqref{IM-thm} for symmetric
power $L$-functions, under certain analytical conditions, in the level aspect.
We will show the full analogue of \eqref{IM-thm} only in the case of symmetic square
$L$-functions, but we will also deduce a corollary  in a very special case for general 
symmetric power $L$-functions.

In order to state the results, now we prepare the notations.
Let $f$ be a primitive form of weight $k$ and level $N$, 
which means that it is a normalized common Hecke-eigen newform
of weight $k$ for $\Gamma_0(N)$. 
We denote by $S_k(N)$ the set of all cusp forms of weight $k$ and level $N$. 
We have a Fourier expansion of $f$ at infinity of the form
\[
f(z)=\sum_{n=1}^{\infty}\lambda_f(n)n^{(k-1)/2}e^{2\pi inz},
\]
where $\lambda_f(1)=1$ and 
the Fourier coefficients $\lambda_f(n)$ are real numbers. 
We consider the $L$-function 
\[
L(f, s)
=
\sum_{n=1}^{\infty}\frac{\lambda_f(n)}{n^s} 
\]
associated with $f$ where $s=\sigma+i\tau\in\mathbb{C}$.
This is absolutely convergent when $\sigma >1$,
and can be continued to the whole of
$\mathbb{C}$ as an entire function.
\par
We denote by $\mathbb{P}$ the set of all prime numbers.
For $N$,
let $\mathbb{P}_N$ be the set of primes which do not divide $N$.
We know that $L(f, s)$ $(\sigma>1)$ has the Euler product expansion
\begin{align*}
L(f, s)
=&
\prod_{p \mid N}(1-\lambda_f(p) p^{-s})^{-1}
\prod_{p \in \mathbb{P}_N}(1-\lambda_f(p) p^{-s} + p^{-2s})^{-1}
\\
=&
\prod_{p \mid N}(1-\lambda_f(p) p^{-s})^{-1}
\prod_{p \in \mathbb{P}_N}(1-\alpha_f(p) p^{-s})^{-1}(1-\beta_f(p) p^{-s})^{-1},
\end{align*}
where 
$\alpha_f(p)$ and $\beta_f(p)$ satisfy 
$\alpha_f(p)+\beta_f(p)=\lambda_f(p)$ and
$|\alpha_f(p)|=|\beta_f(p)|=1$,
and they are the complex conjugate of each other.
This Euler product is deduced from the relations
\begin{equation}\label{euler}
\lambda_f(p^{\ell})=
\begin{cases}
\displaystyle \lambda_f^{\ell}(p) & p \mid N,\\
\displaystyle \sum_{h=0}^{\ell}\alpha_f^{\ell-h}(p)\beta_f^h(p) & p \nmid N.
\end{cases}
\end{equation}
From this relations, we see that $|\lambda_f(n)|\leq d(n)$, where $d(n)$ is
the number of divisors of $n$.
Further, for $p|N$, we have $\lambda_f^2(p)=p^{-1}$ if the $p$-component of $N$ is
just $p$, and $\lambda_f(p)=0$ if $p^2|N$ (see Miyake \cite{mi}).
\par
For any positive integer $r$ and $\sigma>1$,
we denote the (partial) $r$th symmetric power $L$-function by
\[
L_{\mathbb{P}_N}(\mathrm{Sym}_f^{r}, s)=\prod_{p\in \mathbb{P}_N}\prod_{h=0}^{r}(1-\alpha_f^{r-h}(p)\beta_f^h(p) p^{-s})^{-1}.
\]
When $r=1$, this is nothing but the partial $L$-function
$$
L_{\mathbb{P}_N}(f, s)=
\prod_{p \in \mathbb{P}_N}(1-\alpha_f(p) p^{-s})^{-1}(1-\beta_f(p) p^{-s})^{-1}.
$$
Let
\begin{align*}
\log L_{\mathbb{P}_N}(\mathrm{Sym}_f^{r}, s)
=&
-
\sum_{p \in \mathbb{P}_N}
\sum_{h=0}^{r}
\mathrm{Log}(1-\alpha_f^{r-h}(p)\beta_f^h(p)p^{-s}),
\end{align*}
where $\mathrm{Log}$ means the principal branch.
%
%
\par
We assume the following (see \cite{mu-ropar}).
\begin{assumption}[Analytical conditions]
For any primitive form $f\in S_k(N)$, $L_{\mathbb{P}_N}(\mathrm{Sym}_f^r, s)$
can be analytically continued to an entire function.
There are predicted local factors $L_p(\mathrm{Sym}_f^{r}, s)$ for $p\mid N$
and $L_{\mathbb{P}_N}(\mathrm{Sym}_f^{r}, s)$ satisfies the functional equation
\begin{align}\label{6-5}
\Lambda(\mathrm{Sym}_f^{r},s)=\varepsilon_{r,f}
\Lambda(\mathrm{Sym}_f^{r},1-s),
\end{align}
where $|\varepsilon_{r,f}|=1$ and
\begin{align*}
\Lambda(\mathrm{Sym}_f^{r},s)
=&q_{r,f}^{s/2}\widetilde{\Gamma}_{r}(s)L(\mathrm{Sym}_f^{r},s)
\\
L(\mathrm{Sym}_f^{r},s)
=&L_{\mathbb{P}_N}(\mathrm{Sym}_f^{r},s)\prod_{p\mid N} L_p(\mathrm{Sym}_f^{r},s)
\end{align*}
with the conductor $q_{r,f}$ and the ``gamma factor'' $\widetilde{\Gamma}_{r}(s)$.
Here, the gamma factor is written by
\begin{align}\label{gamma-factor}
\widetilde{\Gamma}_{r}(s)
=\pi^{-(r+1)s/2}\prod_{j=1}^{r+1}\Gamma\bigg(\frac{s+\kappa_{j,r}}{2}\bigg),
\end{align}
where $\kappa_{j,r}\in\mathbb{R}$, and each local factor for $p|N$ is written as
\begin{align}\label{local-factor}
L_p(\mathrm{Sym}_f^{r},s)=(1-\lambda_{p,r,f}p^{-s})^{-1},
\quad |\lambda_{p,r,f}|\leq p^{-r/2}
\end{align}
(see Cogdell and Michel \cite{CM04}, Moreno and Shahidi \cite{MS85}, Rouse \cite{rouse07},
and Rouse and Thorner \cite{RT17}).
Moreover, it satisfies a estimate
\begin{align}\label{assump_est}
|L(\mathrm{Sym}_f^{r},s)| \ll_r N(|t|+2)\qquad(1/2 < \sigma \leq 2),
\end{align}
where $\ll$ stands for the Vinogradov symbol, the same as Landau's $O$-symbol.
(The suffix here means that the constant implied by $\ll$ depends on $r$.
In what follows, similarly, the implied constants do not depend on parameters 
which are not
written explicitly, unless otherwise indicated.) 
\end{assumption}
\begin{rem}\label{assump-1-2}
It is known that the statements in the above assumption are true for $r=1$
(classical), or for $r=2$ and $N$ is square-free (Shimura \cite{sh},
Gelbart and Jacquet \cite{gj}).
\end{rem}
In the present paper, we further assume the generalized Riemann hypothesis.
\begin{assumption}[GRH]
  Let $f$ be a primitive form of weight $k$ with $2\leq k < 12$ or $k=14$
and level $q^m$, where $q$ is a prime number.
The $L$-functions $L(\mathrm{Sym}_f^r, s)$
satisfies the Generalized Riemann Hypothesis
which means that $L(\mathrm{Sym}_f^rf, s)$ has no zero
in the strip $1/2 < \sigma \leq 1$.
\end{assumption}
%
%
\par
We consider the primitive forms $f\in S_k(q^m)$,
where $q$ is a prime number and $m$ is a positive integer.
Our aim is to study certain averages of the value of
\[
\log L_{\mathbb{P}_q}(\mathrm{Sym}_f^{r}, \sigma)
=
\log L_{\mathbb{P}_{q^m}}(\mathrm{Sym}_f^r, \sigma)
\]
for $\sigma >1/2$.    Because of Assumption (GRH), there is no problem how to choose
the branch of the logarithm.
Let $\Psi$ be a $\mathbb{C}$-valued function defined on $\mathbb{R}$.
We consider the average
\begin{align}\label{average_def}
&
\Avg \Psi(\log L_{\mathbb{P}_q}(\mathrm{Sym}_f^{r},\sigma))
\nonumber\\
=&
\lim_{q^m\to\infty}
\frac{1}{C_k(1-C_q(m))}\sum_{\substack{f\in S_k(q^m) \\ f\text{:\;primitive\;form}}}\frac{1}{\langle f,f \rangle}
\Psi(\log L_{\mathbb{P}_q}(\mathrm{Sym}_f^{r},\sigma)),
\end{align}
where $\langle,\rangle$ denotes the Petersson inner product, 
\[
C_k=\frac{(4\pi)^{k-1}}{\Gamma(k-1)},\quad
C_q(m)
=
\begin{cases}
  0 & m=1,\\
  q(q^2-1)^{-1} & m=2,\\
  q^{-1} & m\geq 3,
\end{cases}
\]
and the symbol $q^m\to\infty$ means that
(i) $m$ tends to $\infty$ while primes $q$ are bounded, or (ii)
$q$ tends to $\infty$ while $m$ are bounded.
We abbreviate the right-hand side of \eqref{average_def} as
\[
\lim_{q^m\to\infty}
\sum_{f\in S_k(q^m)}^{\qquad\prime}
\Psi(\log L_{\mathbb{P}_q}(\mathrm{Sym}_f^{r},\sigma)).
\]
%
%
%
\par
Let
\[
\rho
=
\begin{cases}
  [r/2] & r\;\text{is odd}\\
  [r/2]-1 & r\;\text{is even}
\end{cases}
=
\begin{cases}
  \displaystyle \frac{r-1}{2} & r\;\text{is odd},\\
  \displaystyle \frac{r}{2}-1 & r\;\text{is even}.
\end{cases}
\]
Our ultimate objective is to construct the ``$M$-function''
$\mathcal{M}_{\sigma}(\mathrm{Sym}^{\rho}, u)$ which satisfies
\begin{align}\label{ultimate}
\Avg \Psi (\log L_{\mathbb{P}_q}(\mathrm{Sym}_f^{r},\sigma))
=
\int_{\mathbb{R}}\mathcal{M}_{\sigma}(\mathrm{Sym}^{\rho}, u)\Psi (u)
\frac{du}{\sqrt{2\pi}}
\end{align}
for any function $\Psi:\mathbb{R}\to\mathbb{C}$
of exponential growth under the GRH. 
This type of result was established in \cite[Theorem 4]{im-moscow}
for Dirichlet $L$-functions of global fields.
We may expect that the
same type of result would hold for symmetric power $L$-functions.
%
%
\par
In the present paper we will establish the following
theorem, which treats the cases $r=1$ or $2$ which means $\rho=0$.
\begin{thm}\label{main1}
Let
$k$ an even integer which satisfies $2\leq k <12$ or $k=14$.
For $r=1, 2$, we
suppose Assumption (Anatytical conditions) and Assumption (GRH).
Then, for $\sigma>1/2$, there exists a function 
$\mathcal{M}_{\sigma}(\mathrm{Sym}^0, u):\mathbb{R}\to\mathbb{R}_{\geq 0}$ 
which can be explicitly constructed, and for which the formula
\begin{align}\label{main-formula}
\Avg \Psi (\log L_{\mathbb{P}_q}(\mathrm{Sym}_f^r,\sigma))
=
\int_{\mathbb{R}}\mathcal{M}_{\sigma}(\mathrm{Sym}^0, u)\Psi (u)
\frac{du}{\sqrt{2\pi}}
\qquad (r=1,2) 
\end{align}
holds for the function $\Psi:\mathbb{R} \to \mathbb{C}$
which is a bounded continuous function
or a compactly supported Riemann integrable function.
\end{thm}
\begin{rem}\label{abs-conv-case}
The function $\mathcal{M}_{\sigma}(\mathrm{Sym}^0, u)$ tends to $0$ as $|u|\to\infty$
(Proposition \ref{M} (3) in Section \ref{sec3}).
In particular, if $\sigma>1$, then $\mathcal{M}_{\sigma}(\mathrm{Sym}^0, u)$ is compactly supported, and hence \eqref{main-formula} is valid for any continuous $\Psi$.
\end{rem}
We will prove Theorem \ref{main1} and Remark \ref{abs-conv-case} at the end of
Section \ref{sec4}.
The main point of this theorem is the case $r=2$, that is the case of
symmetric square $L$-functions.
The case $r=1$ is nothing but the usual $L$-function attached to $f$,
and in this case Mine \cite{mine} proved a stronger unconditional result.
Nevertheless we include the case $r=1$ in the theorem,
because it can be treated in parallel with the case $r=2$,
and also, it is necessary to prove the next corollary. 
Moreover, it is to be stressed that the same ``$M$-function'' 
$\mathcal{M}_{\sigma}(\mathrm{Sym}^0, u)$ appears in \eqref{main-formula} for 
both the cases $r=1$ and $r=2$.
%
%
\begin{cor}\label{main2}
Let 
$k$ an even integer which satisfies $2\leq k <12$ or $k=14$.
Let $r$ be any positive integer.
Suppose Assumption (Anatytical conditions) and Assumption (GRH).
Let $\Psi_1(x)=cx$ with a constant $c$.
If \cite[Theorem~1.5]{mu} and
the above Theorem \ref{main1} are valid for $\Psi=\Psi_1$ (see Remark \ref{Psi_1-valid}
below), 
then, for $\sigma>1/2$, there exists a function 
$\mathcal{M}_{\sigma}^*:\mathbb{R}\to\mathbb{R}_{\geq 0}$ 
which can be explicitly constructed, and for which the formula
\[
\Avg \Psi (\log L_{\mathbb{P}_q}(\mathrm{Sym}_f^r,\sigma))
=
\int_{\mathbb{R}}\mathcal{M}_{\sigma}(\mathrm{Sym}^0, u)\Psi_1 (u)
\frac{du}{\sqrt{2\pi}}
+
\rho\int_{\mathbb{R}}\mathcal{M}_{\sigma}^{\ast}(u)\Psi_1 (u)
\frac{du}{\sqrt{2\pi}}
\]
holds.
\end{cor}

\begin{rem}\label{Psi_1-valid}
In the above corollary we assume that Theorem \ref{main1}
and \cite[Theorem~1.5]{mu} are valid for
$\Psi=\Psi_1$.    Because of Remark \ref{abs-conv-case} (and the corresponding fact
for \cite[Theorem~1.5]{mu}), this assumption is indeed true for $\sigma>1$.
When $1/2<\sigma\leq 1$, we cannot prove this assumption at present, but it is
plausible in view of the results in \cite{im-moscow}.
\end{rem}
\begin{rem}
Corollary \ref{main2} suggests that the $M$-function 
$\mathcal{M}_{\sigma}(\mathrm{Sym}^{\rho}, u)$ for $\rho\geq 1$ may be given by
\[
\mathcal{M}_{\sigma}(\mathrm{Sym}^{\rho}, u)
=
\frac{1}{\rho+1}
\big(
\mathcal{M}_{\sigma}(\mathrm{Sym}^0, u)
+
\rho\mathcal{M}_{\sigma}^{\ast}(u)
\big).
\]
\end{rem}
%
%
\begin{rem}
In \cite{im_2011} and \cite{mu}, the function $\Psi$ was assumed to be a bounded continuous function or a compactly supported characteristic function. But the second-named author \cite{matsumoto}
noticed that ``compactly supported characteristic function'' should be replaced by
``compactly supported Riemann integrable function''.
\end{rem}
%
%
We mention a corollary of Theorem~\ref{main1} and Corollary~\ref{main2}.
For other averages
\[
\Avg_m\Psi (\log L_{\mathbb{P}_q}(\mathrm{Sym}_f^{r},\sigma))
=
\lim_{X \to \infty}
\frac{1}{\pi(X)} \sum_{\substack{q \leq X\\ q :\mathrm{prime}\\ m : \mathrm{fixed}}} 
\sum_{f \in S_k(q^m)}^{\qquad\prime}
\Psi (\log L_{\mathbb{P}_q}(\mathrm{Sym}_f^{r}, \sigma))
\]
where $\pi(X)$ denotes the number of prime numbers not larger than $X$, and
\begin{align*}
&\Avg^* \Psi (\log L_{\mathbb{P}_q}(\mathrm{Sym}_f^{r},\sigma))
\\
=&
\lim_{X \to \infty}
\frac{1}{\pi^*(X)}
\underset{q^m \leq X}{\sum_{q :\text{prime}}\sum_{1 \leq m}}
\sum_{f \in S_k(q^m)}^{\qquad\prime}
\Psi (\log L_{\mathbb{P}_q}(\mathrm{Sym}_f^{r},\sigma)),
\end{align*}
where $\pi^*(X)$ denotes the number of all pairs $(q,m)$
of a prime number $q$ and a positive integer $m$ with $q^m\leq X$.
Theorem~\ref{main1} imply the following corollary.
The proof is same as that in \cite{mu}.
\begin{cor}\label{main-cor}
Under the same assumptions of Theorem~\ref{main1} and Corollary~\ref{main2},
we have
\begin{align*}
&\Avg_m \Psi (\log L_{\mathbb{P}_q}(\mathrm{Sym}_f^{r}, \sigma))
=
\Avg^* \Psi (\log L_{\mathbb{P}_q}(\mathrm{Sym}_f^{r}, \sigma))
\\
=&
\Avg\Psi (\log L_{\mathbb{P}_q}(\mathrm{Sym}_f^{r}, \sigma)).
\end{align*}
\end{cor}
%
%
\section{Preparation}\label{sec2}
\par
In the proof of our main Theorem~\ref{main1}, we will use the following
formula (\eqref{P} below) for a prime number $q$, 
which was shown by the third-named author~\cite[Lemma~3]{ichihara},
using Petersson's formula.
\par
When $2\leq k < 12$ or $k=14$,
for the primitive form $f$ of weight $k$ and level $q^m$, 
we have
\begin{equation}\label{P}
\sum_{f\in S_k(q^m)}^{\qquad\prime}
\lambda_f(n)
=
\delta_{1, n}
+
\begin{cases} 
O(n^{(k-1)/2} q^{-k+1/2})
& m=1, \\
O(n^{(k-1)/2} q^{m(-k+1/2)}q^{k-3/2})
& m\geq 2, 
\end{cases}
\end{equation}
where $\delta_{1, n}=1$ if $n=1$ and $0$ otherwise.

\begin{rem}
The implied constants on the right-hand side depends on $k$, but in the presnt paper 
we restrict $k$ to $2\leq k<12$ or $k=14$, so we may say that the implied constants
above are absolute.
This restriction of the range of $k$ is coming from the matter how to construct 
the basis of the space of old forms in \cite{ichihara} (see
\cite[Remarks 1 and 4]{ichihara}).
\end{rem}

We denote the error term in \eqref{P} by
$n^{(k-1)/2} E(q^m)$, that is
\[
\sum_{f\in S_k(q^m)}^{\qquad\prime}
\lambda_f(n)
-
\delta_{1, n}
=
n^{(k-1)/2}E(q^m).
\]
Then we have
\begin{equation}\label{E1}
E(q^m)\ll q^{-k+1/2}
\end{equation}
for any $m$, 
and 
\begin{align}\label{E2}
E(q^m)\ll& \begin{cases}
          q^{-3/2} & m=1,\\
          q^{-5/2} & m=2,\\
         q^{-1-m}   & m\geq 3,
          \end{cases}
\nonumber \\
\ll& q^{-m-1/2}
\end{align}
for any $m \geq 1$.
Also in the case $n=1$, the formula \eqref{P} implies
\begin{equation}\label{P1}
\sum_{f\in S_k(q^m)}^{\qquad\prime}
\lambda_f(1)
=
\sum_{f\in S_k(q^m)}^{\qquad\prime}
1
=
1+E(q^m) \ll 1.
\end{equation}

Let $\mathcal{P}_q$ be a finite subset of $\mathbb{P}_q$.
For a primitive form $f$ of weight $k$ and level $q^m$,
we define
\[
L_{\mathcal{P}_q}(\mathrm{Sym}_f^{r}, s)
=
\prod_{p \in \mathcal{P}_q}
\prod_{h=0}^{r}
(1-\alpha_f^{r-h}(p)\beta_f^h(p)p^{-s})^{-1}
\]
for $\sigma > 1/2$.
When $r$ is even, this can be written as
\begin{align*}
L_{\mathcal{P}_q}(\mathrm{Sym}_f^{r}, s)
=
\prod_{p \in \mathcal{P}_q}
\Big(\prod_{\substack{h=0 \\ r-2h\neq 0}}^{r}
(1-\alpha_f^{r-h}(p)\beta_f^h(p)p^{-s})^{-1}
\Big)
(1-p^{-s})^{-1}.
\end{align*}
Since $\alpha_f^{r-h}(p)\beta_f^h(p)=\alpha_f^{r-2h}(p)=\beta_f^{2h-r}(p)$,
using $\rho$ we may write
\begin{align}\label{alternative-exp}
\lefteqn{L_{\mathcal{P}_q}(\mathrm{Sym}_f^{r}, s)}\notag\\
&=
\prod_{p\in\mathcal{P}_q}\prod_{h=0}^{\rho}(1-\alpha_f^{r-2h}(p)p^{-s})^{-1}
(1-\beta_f^{r-2h}(p)p^{-s})^{-1}(1-\delta_{r,\text{even}}p^{-s})^{-1},
\end{align}
where $\delta_{r,\text{even}}=1$ if $r$ is even, and $0$ otherwise.

Let $\mathcal{T}_{\mathcal{P}_q}=\prod_{p\in \mathcal{P}_q} \mathcal{T}$
with $\mathcal{T}=\{t\in\mathbb{C}\mid |t|=1\}$.
For a fixed $\sigma>1/2$ and $t_p\in\mathcal{T}$
we define 
\begin{align*}
g_{\sigma, p}(t_p)=&-\log (1-t_pp^{-\sigma})
\\
\mathscr{G}_{\sigma, p}(t_p)
=&
g_{\sigma, p}(t_p)+g_{\sigma, p}(\overline{t_p})+g_{\sigma, p}(\delta_{r,\text{even}})
\\
=&
-2\log|1-t_p p^{-\sigma}|-\log (1-\delta_{r,\text{even}}p^{-\sigma}),
\end{align*}
where $\overline{t_p}$ is the complex conjugate of $t_p$.
Here $g_{\sigma, p}(t_p)$ is the same symbol as in \cite{mu}.
Putting $t_p=e^{i\eta_p}$, we may also write
\begin{align}\label{cos-expression}
\mathscr{G}_{\sigma, p}(e^{i\eta_p})=
-\log(1-2 p^{-\sigma}\cos\eta_p+p^{-2\sigma})-\log (1-\delta_{r,\text{even}}p^{-\sigma}).
\end{align}
Further, for 
$t_{\mathcal{P}_q}=(t_p)_{p\in\mathcal{P}_q} \in\mathcal{T}_{\mathcal{P}_q}$, let
\begin{align*}
g_{\sigma, \mathcal{P}_q}(t_{\mathcal{P}_q})
=&
\sum_{p\in\mathcal{P}_q} g_{\sigma, p}(t_p),
\\
g_{\sigma, \mathcal{P}_q}(\delta_{r,\text{even}})
=&
\sum_{p\in\mathcal{P}_q} 
g_{\sigma, p}(\delta_{r,\text{even}}),
\\
\mathscr{G}_{\sigma, \mathcal{P}_q}(t_{\mathcal{P}_q})
=&
\sum_{p\in\mathcal{P}_q} \mathscr{G}_{\sigma, p}(t_p)
=
\sum_{p\in\mathcal{P}_q}
g_{\sigma, p}(t_p)
+
\sum_{p\in\mathcal{P}_q}
g_{\sigma, p}(\overline{t_p})
+
\sum_{p\in\mathcal{P}_q} g_{\sigma, p}(\delta_{r,\text{even}})
\\
=&
g_{\sigma, \mathcal{P}_q}(t_{\mathcal{P}_q})
+
g_{\sigma, \mathcal{P}_q}(\overline{t_{\mathcal{P}_q}})
+
g_{\sigma, \mathcal{P}_q}(\delta_{r,\text{even}}).
\end{align*}
%
Using \eqref{alternative-exp},
for any $\sigma >1/2$ we have 
\begin{align*}
&
\log L_{\mathcal{P}_q} (\mathrm{Sym}_f^{r}, \sigma)
\nonumber\\
=&
\sum_{p \in \mathcal{P}_q}
\sum_{h=0}^{\rho}
\Big(
-\log(1-\alpha_f^{r-2h}(p)p^{-\sigma})
-\log(1-\beta_f^{r-2h}(p)p^{-\sigma})
\Big)
\nonumber\\
&
-
\sum_{p\in\mathcal{P}_q}
\log(1-\delta_{r,\text{even}}p^{-\sigma})
\nonumber\\
=&
\sum_{p \in \mathcal{P}_q}
\sum_{h=0}^{\rho}
\Big(
g_{\sigma, p}(\alpha_f^{r-2h}(p))
+g_{\sigma, p}(\beta_f^{r-2h}(p))
\Big)
+
\sum_{p\in\mathcal{P}_q}
g_{\sigma, p}(\delta_{r,\text{even}})
\nonumber\\
=&
\sum_{h=0}^{\rho}
\Big(
g_{\sigma, \mathcal{P}_q}(\alpha_f^{r-2h}(\mathcal{P}_q))
+g_{\sigma, \mathcal{P}_q}(\beta_f^{r-2h}(\mathcal{P}_q))
\Big)
+
g_{\sigma, \mathcal{P}_q}(\delta_{r,\text{even}}),
\end{align*}
where
$\alpha_f^{\mu}(\mathcal{P}_q)=(\alpha_f^{\mu}(p))_{p\in\mathcal{P}_q}$
and $\beta_f^{\mu}(\mathcal{P}_q)=(\beta_f^{\mu}(p))_{p\in\mathcal{P}_q}$.
Especially, in the case $r=1, 2$ which means $\rho=0$, we have
\begin{equation}\label{g_sigma}
\log L_{\mathcal{P}_q} (\mathrm{Sym}_f^{r}, \sigma)
=
\mathscr{G}_{\sigma, \mathcal{P}_q}(\alpha_f^r(\mathcal{P}_q)).
\end{equation}
We sometimes write $\alpha_f(p)=e^{i\theta_f(p)}$ and  
$\beta_f(p)=e^{-i\theta_f(p)}$ for $\theta_f(p)\in [0,\pi]$.
\par
In the case $\rho=0$ and $\sigma>1$, 
we deal with the value  
$L_{\mathbb{P}_q}(\mathrm{Sym}_f^{r}, s)$
as the limit of the value  
$L_{\mathcal{P}_q}(\mathrm{Sym}_f^{r}, s)$
as $\mathcal{P}_q$ tends to $\mathbb{P}_q$.
In fact, from \eqref{g_sigma} we have 
\begin{equation}\label{g_sigma2}
\log L_{\mathbb{P}_q}(\mathrm{Sym}_f^{r}, \sigma)
=
\lim_{\substack{\mathcal{P}_q\to\mathbb{P}_q\\ \mathcal{P}_q\subset {\mathbb{P}_q}}}
\sum_{h=0}^{\rho}
\mathscr{G}_{\sigma, \mathcal{P}_q}(\alpha_f^{r-2h}(\mathcal{P}_q))
\end{equation}
In the case $\rho=0$ and $1\geq \sigma>1/2$, 
we will prove the relation between
$\log L_{\mathbb{P}_q}(\mathrm{Sym}_f^{r}, \sigma)$
and
$\log L_{\mathcal{P}_q}(\mathrm{Sym}_f^{r}, \sigma)$
with a suitable finite subset $\mathcal{P}_q \subset \mathbb{P}_q$ 
depending on $q^m$ and will consider the averages of them.

In Sections~3, 4 and 5, we will prove Theorem~\ref{main1}.
The proof of Corollary~\ref{main2} will be completed in Section 6.
%
%
\section{The density function ${\mathcal{M}}_{\sigma}(\mathrm{Sym}^0, u)$ and its Fourier 
transform}\label{sec3}
\par
In this section we first construct the density function 
${\mathcal{M}}_{\sigma, \mathcal{P}_q}(\mathrm{Sym}^0, u)$
for a finite set $\mathcal{P}_q \subset \mathbb{P}_q$
in the case $r=1, 2$.
By $|\mathcal{P}_q|$ we denote the number of the elements of $\mathcal{P}_q$.
Let
\[
\Theta_{\mathcal{P}_q}
=\prod_{p\in\mathcal{P}_q} [0, \pi),
\]
and define the modified Sato-Tate measure on $\Theta_{\mathcal{P}_q}$ by
\[
d^{\rm ST} \theta_{\mathcal{P}_q}
=\prod_{p\in\mathcal{P}_q}
\bigg(\frac{2\sin^2 \theta_p}{\pi} d\theta_p\bigg),
\]
where
$\theta_{\mathcal{P}_q}=(\theta_p)_{p\in\mathcal{P}_q}\in \Theta_{\mathcal{P}_q}$.
We also define the normalized 
Haar measure on $\mathcal{T}_{\mathcal{P}_q}$ by
\[
d^H t_{\mathcal{P}_q}
=
\prod_{p\in\mathcal{P}_q} d^H t_p
=
\prod_{p\in\mathcal{P}_q} 
 \frac{dt_p}{2\pi it_p}.
 \]

The following proposition is an analogue of \cite[Proposition 3.1]{mu}, but it is a
big difference that here we use the modified Sato-Tate measure.

%
%
\begin{prop}~\label{M_P}
Let $r=1$ or $2$ {\rm(}and hence $\rho=0${\rm)}.
For any $\sigma >1/2$, there exists a $\mathbb{R}$-valued, non-negative function
$\mathcal{M}_{\sigma, \mathcal{P}_q}(\mathrm{Sym}^0, u)$ defined on $\mathbb{R}$
which satisfies following two properties.
\begin{itemize}
\item The support of $\mathcal{M}_{\sigma, \mathcal{P}_q}(\mathrm{Sym}^0, u)$ is compact.
\item For any continuous function $\Psi$ on $\mathbb{R}$,
we have
\begin{align}\label{prop-1-1}
&\int_{\mathbb{R}}
\mathcal{M}_{\sigma, \mathcal{P}_q}(\mathrm{Sym}^0, u) \Psi(u)
\frac{du}{\sqrt{2\pi}}
=
\int_{\Theta_{\mathcal{P}_q}}
\Psi\big(
\mathscr{G}_{\sigma, \mathcal{P}_q} (e^{i\theta_{\mathcal{P}_q}r})
\big)
d^{\rm ST} \theta_{\mathcal{P}_q}\notag\\
&\qquad=
\int_{\mathcal{T}_{\mathcal{P}_q}} 
\Psi(
\mathscr{G}_{\sigma, \mathcal{P}_q}( t_{\mathcal{P}_q}^r)
)
\prod_{p\in\mathcal{P}_q}
\bigg(\frac{t_p^2 -2 + t_p^{-2}}{-2}\bigg)
d^H t_{\mathcal{P}_q},
\end{align}
\end{itemize}
In particular, taking $\Psi\equiv 1$ in \eqref{prop-1-1}, we have
\begin{align}\label{prop-1-2}
\int_{\mathbb{R}} \mathcal{M}_{\sigma, \mathcal{P}_q}(\mathrm{Sym}^0, u) 
\frac{du}{\sqrt{2\pi}}=1. 
\end{align}
\end{prop}
\begin{proof}
We construct the function $\mathcal{M}_{\sigma, \mathcal{P}_q}(\mathrm{Sym}^0, u)$
by using the method similar to that in \cite{im_2011}.
In the case $|\mathcal{P}_q|=1$ namely $\mathcal{P}_q=\{p\}$, 
we define a one-to-one correspondence $\theta_p\mapsto u$ from the interval
$[0, \pi)$ to another interval
\begin{align*}
&
A(\sigma, p)
\\
&=
\left(-2\log (1+p^{-\sigma})-\log(1-\delta_{r,\text{even}}p^{-\sigma}), -2\log (1-p^{-\sigma})-\log(1-\delta_{r,\text{even}}p^{-\sigma})\right]\\
&\subset \mathbb{R}
\end{align*}
by 
\begin{align*}
u
=&
-\log(1-2p^{-\sigma}\cos \theta_p+p^{-2\sigma})
-\log(1-\delta_{r,\text{even}}p^{-\sigma})
\\
=&
\begin{cases}
-2\log |1-e^{i\theta_p}p^{-\sigma}| & r=1
\\
-\log(1-(2\cos\theta_p+1)p^{-\sigma}+(2\cos\theta_p+1)p^{-2\sigma}
-p^{-3\sigma}) & r=2.
\end{cases}
\end{align*}
%
%
The definition of $\mathcal{M}_{\sigma, \{p\}}(\mathrm{Sym}^0, u)=\mathcal{M}_{\sigma, p}(\mathrm{Sym}^0, u)$ is 
\[
\mathcal{M}_{\sigma, p}(\mathrm{Sym}^0, u)
=
\begin{cases}
  \displaystyle
\frac{\sqrt{2\pi}|1-e^{i\theta_p}p^{-\sigma}|^2\sin^2(\theta_p/r)}{-\pi p^{-\sigma}\sin\theta_p}
  & \displaystyle u \in A(\sigma,  p),\\
  0 & \text{otherwise},
\end{cases}
\]
where we see that
\[
\frac{\sqrt{2\pi}|1-e^{i\theta_p}p^{-\sigma}|^2\sin^2(\theta_p/r)}{-\pi p^{-\sigma}\sin\theta_p}
=
\begin{cases}
  \displaystyle \frac{\sqrt{2\pi}|1-e^{i\theta_p}p^{-\sigma}|^2\sin \theta_p}{-\pi p^{-\sigma}} & r=1\\
  \displaystyle \frac{\sqrt{2\pi}|1-e^{i\theta_p}p^{-\sigma}|^2\tan(\theta_p/2)}{-2 \pi p^{-\sigma}} & r=2
\end{cases}
\]
for $\theta_p\in [0, \pi)$.    

We show that
this function satisfies the properties required by Proposition~\ref{M_P}.
In fact, using 
\[
  \frac{du}{d\theta_p}
  =
    -\frac{2p^{-\sigma} \sin \theta_p}{1-2p^{-\sigma}\cos \theta_p +p^{-2\sigma}}
  =
   -\frac{2p^{-\sigma} \sin \theta_p}{|1-e^{i \theta_p}p^{-\sigma}|^2},
\]
we obtain 
\begin{align}\label{st_p0}
&
\int_{\mathbb{R}}
\Psi(u)    
\mathcal{M}_{\sigma, p}(\mathrm{Sym}^0, u) \frac{du}{\sqrt{2\pi}}
=
\int_{A(\sigma, p)}
\Psi(u) \mathcal{M}_{\sigma, p}(\mathrm{Sym}^0, u) \frac{du}{\sqrt{2\pi}}
\nonumber\\
=&
\int_0^{\pi}
\Psi\Big(-\log(1-2p^{-\sigma}\cos \theta_p +p^{-2\sigma})
-\log(1-\delta_{r,\text{even}}p^{-\sigma})
\Big)
\nonumber\\
&\times
\frac{|1-e^{i\theta_p}p^{-\sigma}|^2\sin^2(\theta_p/r)}{-\pi p^{-\sigma}\sin\theta_p}
\frac{-2p^{-\sigma} \sin \theta_p}{|1-e^{i \theta_p}p^{-\sigma}|^2}
d\theta_p
\nonumber\\
=&
\int_0^{\pi}
\Psi\Big(-\log(1-2p^{-\sigma}\cos \theta_p +p^{-2\sigma})
-\log(1-\delta_{r,\text{even}}p^{-\sigma})
\Big)
  \frac{2\sin^2(\theta_p/r)}{\pi}
d\theta_p.
\end{align} 
If $r=2$, (putting $\theta_p=2\theta_p^{\prime}$)
we see that the right-hand side is
\begin{align*}
=
\int_0^{\pi/2}
\Psi\Big(-\log(1-2p^{-\sigma}\cos 2\theta_p^{\prime} +p^{-2\sigma})
-\log(1-p^{-\sigma})
\Big)
\frac{4\sin^2\theta_p^{\prime}}{\pi}
d\theta_p^{\prime},
\end{align*}
which is further equal to
\begin{align*}
\int_0^{\pi}
\Psi\Big(-\log(1-2p^{-\sigma}\cos 2\theta_p +p^{-2\sigma})
-\log(1-p^{-\sigma})
\Big)
\frac{2\sin^2\theta_p}{\pi}
d\theta_p,
\end{align*}
because putting $\theta_p^{\prime\prime}=\pi-\theta_p^{\prime}$ we have
\begin{align*}
  &
\int_0^{\pi/2}
\Psi\Big(-\log(1-2p^{-\sigma}\cos 2\theta_p^{\prime} +p^{-2\sigma})
-\log(1-p^{-\sigma})
\Big)
\frac{2\sin^2\theta_p^{\prime}}{\pi}
d\theta_p^{\prime}
\\
=&
\int_{\pi/2}^{\pi}
\Psi\Big(-\log(1-2p^{-\sigma}\cos 2\theta_p^{\prime\prime} +p^{-2\sigma})
-\log(1-p^{-\sigma})
\Big)
\frac{2\sin^2\theta_p^{\prime\prime}}{\pi}
d\theta_p^{\prime\prime}.
\end{align*}
Therefore, for $r=1, 2$, we have
\begin{align}\label{st_p}
&
\int_{\mathbb{R}}
\Psi(u)    
\mathcal{M}_{\sigma, p}(\mathrm{Sym}^0, u) \frac{du}{\sqrt{2\pi}}
\nonumber\\
=&
\int_0^{\pi}
\Psi\Big(-\log(1-2p^{-\sigma}\cos (r\theta_p^{\prime}) +p^{-2\sigma})
-\log(1-\delta_{r,\text{even}}p^{-\sigma})
\Big)
\frac{2\sin^2 \theta_p}{\pi}d\theta_p
\nonumber\\
=&
\int_0^{\pi}
\Psi\big(\mathscr{G}_{\sigma, p}(e^{i \theta_p r})\big)
d^{\rm ST} \theta_p
\end{align}
by \eqref{cos-expression},
where $d^{\rm ST}\theta_p=d^{\rm ST}\theta_{\{p\}}$.

We further calculate the right-hand side of \eqref{st_p}.
Since putting $\eta_p=2\pi-\theta_p$ we have
\begin{align*}
  &
\int_{\pi}^{2\pi}
\Psi\Big(-\log(1-2p^{-\sigma}\cos(r\theta_p) +p^{-2\sigma})
-\log(1-\delta_{r,\text{even}}p^{-\sigma})
\Big)
\frac{2\sin^2\theta_p}{\pi}d\theta_p
\\
=&
\int_{0}^{\pi}
\Psi\Big(-\log(1-2p^{-\sigma}\cos(r\eta_p) +p^{-2\sigma})
-\log(1-\delta_{r,\text{even}}p^{-\sigma})
\Big)
d^{\rm ST}\eta_p,
\end{align*}
we find that the right-hand side of \eqref{st_p} is equal to
\begin{align}\label{haar_p}
\frac{1}{2}
&\int_0^{2\pi}
\Psi\Big(-\log(1-2p^{-\sigma}\cos(r\theta_p) +p^{-2\sigma})
-\log(1-\delta_{r,\text{even}}p^{-\sigma})
\Big)
d^{\rm ST} \theta_p
\nonumber\\
=&
\frac{1}{2}
\int_0^{2\pi}
\Psi\big(\mathscr{G}_{\sigma, p}(e^{ir\theta_p})
\big)
\frac{2}{\pi}\big(\frac{e^{i \theta_p}-e^{-i \theta_p}}{2i}\big)^2
d\theta_p
\nonumber\\
=&
\int_0^{2\pi}
\Psi\big(
\mathscr{G}_{\sigma, p}(e^{ir\theta_p})
\big)
\frac{1}{\pi}\big(\frac{e^{2i\theta_p}-2+e^{-2i\theta_p}}{-4}\big)
d\theta_p
\nonumber\\
=&
\int_{\mathcal{T}}
\Psi\big(\mathscr{G}_{\sigma, p}(t_p^r)\big)
\frac{1}{\pi}\big(\frac{t_p^2-2+t_p^{-2}}{-4}\big)\frac{dt_p}{it_p}
\nonumber\\
=&
\int_{\mathcal{T}}
\Psi\big(\mathscr{G}_{\sigma, p} (t_p^r)\big)
\frac{t_p^2-2+t_p^{-2}}{-2} d^Ht_p.
\end{align}
Moreover, choosing $\Psi\equiv 1$ in \eqref{haar_p}, 
We see that
\[
\int_{\mathbb{R}}
\mathcal{M}_{\sigma, \{p\}}(\mathrm{Sym}^0, u) \frac{du}{\sqrt{2\pi}}
=\int_{\mathcal{T}}\frac{t_p^2-2+t_p^{-2}}{-2} d^Ht_p =1.
\]
Therefore all the assertions of the proposition in the case $|\mathcal{P}_q|=1$
are now established.

In the case $\infty > |\mathcal{P}_q|>1$, 
we construct the function
${\mathcal{M}}_{\sigma, \mathcal{P}_q}(\mathrm{Sym}^0, u)$
by the convolution product of
${\mathcal{M}}_{\sigma, \mathcal{P}'_q}(\mathrm{Sym}^0, u)$
and ${\mathcal{M}}_{\sigma, p}(\mathrm{Sym}^0, u)$
for $\mathcal{P}_q=\mathcal{P}'_q\cup\{p\} \subset\mathbb{P}_q$ inductively,
that is
\[
\mathcal{M}_{\sigma, \mathcal{P}_q}(\mathrm{Sym}^0, u)
=
\int_{\mathbb{R}}
\mathcal{M}_{\sigma, \mathcal{P}'_q}(\mathrm{Sym}^0, u')
\mathcal{M}_{\sigma, p}(\mathrm{Sym}^0, u-u')
\frac{du'}{\sqrt{2\pi}}.
\]
It is easy to show that 
this function satisfies the statements of Proposition~\ref{M_P}.
\end{proof}
%
%
\par
Next, for the purpose of considering 
$\displaystyle \lim_{|\mathcal{P}_q|\to \infty} \mathcal{M}_{\sigma, \mathcal{P}_q}(\mathrm{Sym}^0, u)$, 
we consider the Fourier transform of
$\mathcal{M}_{\sigma, \mathcal{P}_q}(\mathrm{Sym}^0, u)$ for $\sigma >1/2$. 
When $\mathcal{P}_q=\{p\}$, we define
$\widetilde{\mathcal{M}}_{\sigma, \mathcal{P}_q}(\mathrm{Sym}^0, u)=\widetilde{\mathcal{M}}_{\sigma, p}(\mathrm{Sym}^0, u)$
by
\[
\widetilde{\mathcal{M}}_{\sigma, p}(\mathrm{Sym}^0, x)
=
\int_{\mathbb{R}} \mathcal{M}_{\sigma, p}(\mathrm{Sym}^0, u) \psi_x(u)
\frac{du}{\sqrt{2\pi}},
\]
where $\psi_x(u)=e^{ixu}$ for $x\in\mathbb{R}$.
%
From \eqref{st_p0},
we see that
\begin{align}\label{widetilde-M_p}
\widetilde{\mathcal{M}}_{\sigma, p}(\textrm{Sym}^0, x)
=
\frac{2}{\pi} \int_0^{\pi}
e^{ixF(\theta)}
\sin^2(\theta/r)d\theta,
\end{align}
where
$
F(\theta)=-\log(1-2p^{-\sigma}\cos\theta+p^{-2\sigma})-\log(1-\delta_{r,\text{even}}p^{-\sigma}).
$
\begin{lem}\label{lemma-M-est}
For $r=1,2$, the following two estimates hold:
\begin{align}\label{M-est-1}
|\widetilde{\mathcal{M}}_{\sigma, p}(\mathrm{Sym}^0, x)|\leq 1
\end{align}
for any prime $p$, and there exists a large $p_0$ such that
\begin{align}\label{M-est-2}
\widetilde{\mathcal{M}}_{\sigma, p}(\mathrm{Sym}^0, x)\ll
\frac{p^{\sigma}}{\sqrt{1+|x|}}
\end{align}
for any $p>p_0$.
\end{lem}

\begin{proof}
The first inequality is easy, because
\begin{align}\label{trivial_p}
|\widetilde{\mathcal{M}}_{\sigma, p}(\mathrm{Sym}^0, x)|
\leq&
\frac{2}{\pi }
\int_0^{\pi}
\sin^2(\theta/r)
d\theta
\nonumber\\
=&
  \frac{2}{\pi}
  \int_0^{\pi}
  \frac{1-\cos(2\theta/r)}{2}
d\theta
=1,
\end{align}
for $r=1, 2$.

To prove the second inequality, we may assume $|x|\geq 1$, because if $|x|<1$,
then \eqref{M-est-2} follows immediately from \eqref{M-est-1}.
We first observe
\begin{align*}
&
\widetilde{\mathcal{M}}_{\sigma, p}(\mathrm{Sym}^0, x)
=
\frac{2}{\pi}\int_0^{\pi} e^{ixF(\theta)}
\sin^2(\theta/r)
d\theta
\\
=&
\begin{cases}
\displaystyle
\frac{2}{\pi}\int_0^{\pi} (e^{ixF(\theta)})'
\frac{(1-2p^{-\sigma}\cos\theta +p^{-2\sigma})}{-2ix p^{-\sigma}}\sin\theta
d\theta & r=1 \\
\displaystyle
\frac{1}{\pi}\int_{0}^{\pi} e^{ixF(\theta)}(1-\cos\theta)d\theta
& r=2,
\end{cases}
\end{align*}
because
\begin{align*}
F^{\prime}(\theta)=\frac{-2p^{-\sigma}\sin\theta}
{1-2p^{-\sigma}\cos\theta+p^{-2\sigma}}.
\end{align*}
When $r=1$, the above is equal to
$$
\frac{2}{\pi}\int_0^{\pi} e^{ixF(\theta)}
\frac{(\cos\theta-2p^{-\sigma}\cos(2\theta) +p^{-2\sigma}\cos\theta)}{+2ix p^{-\sigma}}
d\theta
$$
by integration by parts, and hence $O(p^{\sigma}|x|^{-1})$ which is sufficient.
When $r=2$, the above is
$$
=\frac{1}{\pi}\int_0^{\pi}e^{ixF(\theta)}d\theta-
\frac{1}{\pi}\int_0^{\pi}e^{ixF(\theta)}\cos\theta d\theta
=\frac{1}{\pi}I_1(x)-\frac{1}{\pi}I_2(x),
$$
say.
The estimate
\begin{align}\label{est-I-1}
I_1(x)\ll p^{\sigma/2}|x|^{-1/2}
\end{align}
follows from the Jessen-Wintner inequality (see \cite[Theorem 13]{jw} 
or \cite[Proposition 7.1]{mu-JNT})
for large $p$.
To evaluate $I_2(x)$, we divide
\begin{align*}
  I_2(x)
  &=
\frac{2}{\pi}\int_{\substack{0\leq \theta \leq \sqrt{|x|}^{-1}\\\pi-\sqrt{|x|}^{-1}\leq \theta \leq \pi}} e^{ixF(\theta)}\cos\theta d\theta
  +
  \frac{2}{\pi}\int_{\sqrt{|x|}^{-1}}^{\pi-\sqrt{|x|}^{-1}} e^{ixF(\theta)}\cos\theta d\theta\\
&= I_{21}(x)+I_{22}(x),
\end{align*}
say.
The inequalties
\begin{align*}
\bigg|\frac{xF'(\theta)}{\cos\theta}\bigg|
=
\bigg|\frac{x2p^{-\sigma}\tan\theta}{1-2p^{-\sigma}\cos\theta +p^{-2\sigma}}\bigg|
\geq
\frac{|x|2p^{-\sigma}\theta}{1+2p^{-\sigma}+p^{-2\sigma}}
\\
\bigg|\frac{xF'(\pi-\theta)}{\cos\theta}\bigg|
=
\bigg|\frac{-x2p^{-\sigma}\tan\theta}{1+2p^{-\sigma}\cos\theta +p^{-2\sigma}}\bigg|
\geq
\frac{|x|2p^{-\sigma}\theta}{1+2p^{-\sigma}+p^{-2\sigma}}
\end{align*}
for $0\leq \theta\leq \pi/2$ yields
\begin{align*}
I_{22}(x)
=&
\frac{2}{\pi}\int_{\sqrt{|x|}^{-1}}^{\pi/2} e^{ixF(\theta)}\cos\theta d\theta
+
\frac{2}{\pi}\int_{\pi/2}^{\pi-\sqrt{|x|}^{-1}} e^{ixF(\theta)}\cos\theta d\theta
\\
=&
\frac{2}{\pi}\int_{\sqrt{|x|}^{-1}}^{\pi/2} e^{ixF(\theta)}\cos\theta d\theta
-
\frac{2}{\pi}\int_{\sqrt{|x|}^{-1}}^{\pi/2} e^{ixF(\pi-\theta)}\cos\theta d\theta
\ll
\frac{p^{\sigma}}{\sqrt{|x|}}
\end{align*}
for large $p$ by the first derivative test (Titchmarsh \cite[Lemma 4.3]{tit}), 
while trivially
\[
|I_{21}(x)|
\leq
\frac{2}{\pi}\int_{\substack{0\leq \theta \leq \sqrt{|x|}^{-1}\\\pi-\sqrt{|x|}^{-1}\leq \theta \leq \pi}} d\theta
\ll
\frac{1}{\sqrt{|x|}}.
\]
Therefore 
\begin{align}\label{est-I-2}
I_2(x)\ll p^{\sigma}|x|^{-1/2}.
\end{align}
The desired estimate follows from \eqref{est-I-1} and \eqref{est-I-2}.
\end{proof}

Now we define $\widetilde{\mathcal{M}}_{\sigma, \mathcal{P}_q}(\mathrm{Sym}^0, x)$
by
\begin{align}\label{FFourier-transf}
\widetilde{\mathcal{M}}_{\sigma, \mathcal{P}_q}(\mathrm{Sym}^0, x)
=
\prod_{p \in \mathcal{P}_q}
\widetilde{\mathcal{M}}_{\sigma, p}(\mathrm{Sym}^0, x).
\end{align}
The estimate \eqref{M-est-1} gives
\begin{equation}\label{trivial_P}
\left|\widetilde{\mathcal{M}}_{\sigma, \mathcal{P}_q}(\mathrm{Sym}^0, x)\right|
\leq
\prod_{p\in\mathcal{P}_q^0}
\frac{2}{\pi}\int_0^{\pi}\sin^2\theta_p d\theta_p
=1.
\end{equation}
for any finite $\mathcal{P}_q$.
On the other hand, let $n\geq 3$, and let
$\mathcal{P}_q^{0,n}\subset\mathbb{P}_q$ be a fixed finite set, 
$|\mathcal{P}_q^{0,n}|=n$, and of which 
all elements are $>p_0$.
If $\mathcal{P}_q\supset\mathcal{P}_q^{0,n}$, then
we apply \eqref{M-est-2} to all $p\in\mathcal{P}_q^{0,n}$ and
apply \eqref{M-est-1} to all the other members of $\mathcal{P}_q$ to obtain
\begin{align}\label{combined}
\widetilde{\mathcal{M}}_{\sigma, \mathcal{P}_q}(\mathrm{Sym}^0, x)
=O_n\left(\frac{\prod_{p\in\mathcal{P}_q^{0,n}}p^{\sigma}}
{(1+|x|)^{n/2}}\right).
\end{align}
In what follows we mainly use the case $n=3$, so we write
$\mathcal{P}_q^{0,3}=\mathcal{P}_q^0$.
Denote the largest element of $\mathcal{P}_q^0$ by $y_0$.

%
%
%
\par
The next lemma is an analogue of (a), (b) and (c) given in
\cite[Section 3, pp. 645--646]{im_2011}.

\begin{lem}\label{lem-abc}
The following assertions hold for any $\sigma >1/2$.
\begin{itemize}
\item[{\rm(a)}.]
  Let $\mathcal{P}_q \subset \mathbb{P}_q$ be a finite subset
  with $\mathcal{P}_q\supset\mathcal{P}_q^0$.
  Then  
  $\widetilde{\mathcal{M}}_{\sigma, \mathcal{P}_q}(\mathrm{Sym}^0, x) \in L^t$ $(t\in [1,\infty])$.
\item[{\rm(b)}.]
  For any subsets $\mathcal{P}''_q$ and $\mathcal{P}'_q$ of $\mathbb{P}_q$
  with $\mathcal{P}''_q\subset \mathcal{P}'_q$, we can see
  \[
  \left|\widetilde{\mathcal{M}}_{\sigma, \mathcal{P}'_q}(\mathrm{Sym}^0, x)\right|
  \leq
  \left|\widetilde{\mathcal{M}}_{\sigma, \mathcal{P}''_q}(\mathrm{Sym}^0, x)\right|.
  \]
\item[{\rm(c)}.]
  Let $y>y_0$,
  and put $\mathcal{P}_q(y)=\{ p \in \mathbb{P}_q \mid  p \leq y \} $.
  {\rm(}Hence $\mathcal{P}_q^0\subset\mathcal{P}_q(y)\subset \mathbb{P}_q$.{\rm)}
  The limit
  $\displaystyle \lim_{y\to\infty} \widetilde{\mathcal{M}}_{\sigma, \mathcal{P}_q(y)}(\mathrm{Sym}^0, x)$ exists, which we 
  denote by $\widetilde{\mathcal{M}}_{\sigma}(\mathrm{Sym}^0, x)$.
  For any $a >0$, this convergence is uniform in $|x|\leq a$.
  The infinite product expression 
  $\widetilde{\mathcal{M}}_{\sigma}(\mathrm{Sym}^0, x)=\prod_{p\in\mathbb{P}_q}\widetilde{\mathcal{M}}_{\sigma, p}(\mathrm{Sym}^0, x)$
  holds, which is absolutely convergent.
\end{itemize}
\end{lem}

\begin{proof}
First prove (a):
By \eqref{trivial_P} and \eqref{combined} (with $n=3$),
we find
\begin{align}\label{lem2-a-pr}
\int_{\mathbb{R}}|\widetilde{\mathcal{M}}_{\sigma, \mathcal{P}_q}(\textrm{Sym}^0, x)|^t|dx|
\ll 
\int_{|x|< 1} |dx|
+
\int_{|x|\geq 1}
\frac{\prod_{p\in\mathcal{P}_q^0}p^{t\sigma}}{|x|^{3t/2}} |dx|
< 
\infty
\end{align}
for $t\in [1, \infty)$.
Note that the implied constant here depends only on $\sigma,t$, and $\mathcal{P}_q^0$.
Also from \eqref{trivial_P},
we see that
$\sup_x |\widetilde{\mathcal{M}}_{\sigma, \mathcal{P}_q}(\mathrm{Sym}^0, x)| < \infty $.
The assertion (a) is proved.
The property (b) is easy by \eqref{trivial_p}.

We proceed to the proof of (c).
For any integers $y_1$ and $y_2$ satisfying $y_2\geq y_1>y_0$,
let $p(i)$ $(1\leq i \leq d)$ be all prime numbers
in $(y_1, y_2]$ with
\[
y_2\geq p(d) >p(d-1) > \ldots > p(1) > y_1.
\]
Set the following subsets of $\mathbb{P}_q$:
\begin{align*}
\mathcal{P}_q(y_1, 0)=&\mathcal{P}_q(y_1),
\\
\mathcal{P}_q(y_1, 1)=&\mathcal{P}_q(y_1, 0)\cup \{p(1)\},
\\
\vdots &
\\
\mathcal{P}_q(y_1, d)=&\mathcal{P}_q(y_1, d-1)\cup \{p(d)\}.
\end{align*}
%
First we note that
\begin{align}\label{wp_00}
\mathscr{G}_{\sigma, p}(e^{i\theta})
&=
-\log(1-e^{ir\theta}p^{-\sigma})-\log(1-e^{-ir\theta}p^{-\sigma})
-\log(1-\delta_{r,\text{even}}p^{-\sigma})\notag\\
&\ll
\frac{1}{p^{\sigma}}.
\end{align}
Next we show that
\begin{align}\label{wp_0}
-\frac{2}{\pi}\int_0^{\pi}
\mathscr{G}_{\sigma,p}(e^{ir\theta})\sin^2\theta
d\theta
\ll 
\frac{1}{p^{2\sigma}}.
\end{align}
In fact, the. left-hand side is
\begin{align}
=&
-\frac{2}{\pi}\int_0^{\pi}
\Big(-\log(1-2p^{-\sigma}\cos(r\theta) +p^{-2\sigma})
-\log(1-\delta_{r,\text{even}}p^{-\sigma})
\Big)
\sin^2{\theta}
d\theta
\nonumber\\
=&
-\frac{1}{\pi}\int_0^{2\pi}
\Big(-\log(1-2p^{-\sigma}\cos(r\theta) +p^{-2\sigma})
-\log(1-\delta_{r,\text{even}}p^{-\sigma})
\Big)
\sin^2{\theta}
d\theta
\nonumber\\
=&
-\frac{1}{\pi}
\int_0^{2\pi}
(-\log(1-e^{ir\theta}p^{-\sigma})-\log(1-e^{-ir\theta}p^{-\sigma})
-\log(1-\delta_{r,\text{even}}p^{-\sigma}))
\nonumber\\
&\times
\bigg(\frac{e^{2i\theta}-2+e^{-2i\theta}}{-4}\bigg)
d\theta
\nonumber\\
=&
\frac{1}{\pi}
\int_0^{2\pi}
\bigg(
\sum_{h=1}^{\infty}\frac{e^{hir\theta}}{hp^{h\sigma}}
+\sum_{k=1}^{\infty}\frac{e^{-k ir\theta}}{k p^{k\sigma}}
+\sum_{\ell=1}^{\infty}\frac{\delta_{r,\text{even}}}{\ell p^{\ell\sigma}}
\bigg)
\bigg(\frac{e^{2i\theta}-2+e^{-2i\theta}}{+4}\bigg)
d\theta.
\nonumber
\end{align}
We compute the right-hand side of the above termwisely.
When $r=1$, only the contributions of the terms corresponding to $h=2$ and $k=2$
remain, and the right-hand side is
$$
=\frac{1}{4\pi}\int_0^{2\pi}\left(\frac{1}{2p^{2\sigma}}+\frac{1}{2p^{2\sigma}}
\right) d\theta=\frac{1}{2p^{2\sigma}}.
$$
When $r=2$, the remaining terms are those corresponding to $h=1$, $k=1$ and
$\ell=1,2,\ldots$, and hence the right-hand side is
\begin{align}\label{kokogadaiji}
=\frac{1}{4\pi}\int_0^{2\pi}
\Big(\frac{1}{p^{\sigma}}+\frac{1}{p^{\sigma}}
-2\sum_{\ell=1}^{\infty}\frac{1}{\ell p^{\ell\sigma}}
\Big)
d\theta
=-\sum_{\ell=2}^{\infty}\frac{1}{\ell p^{\ell\sigma}}
\ll \frac{1}{p^{2\sigma}}.
\end{align}
Therefore we obtain \eqref{wp_0}.
(Here, it is an important point in the proof that the terms of the form
$1/p^{\sigma}$ are finally cancelled with each other in the computations
\eqref{kokogadaiji}.) 

Now we evaluate the following difference:
\begin{align}\label{wp_1}
&
\left|
\widetilde{\mathcal{M}}_{\sigma, \mathcal{P}_q(y_1, d)}(\mathrm{Sym}^0, x)
-
\widetilde{\mathcal{M}}_{\sigma, \mathcal{P}_q(y_1, 0)}(\mathrm{Sym}^0, x)
\right|
\nonumber\\
\leq &
\sum_{i=1}^{d}
\left|
\widetilde{\mathcal{M}}_{\sigma, \mathcal{P}_q(y_1, i)}(\mathrm{Sym}^0, x)
-
\widetilde{\mathcal{M}}_{\sigma, \mathcal{P}_q(y_1, i-1)}(\mathrm{Sym}^0, x)
\right|
\nonumber\\
\leq&
\sum_{i=1}^{d}
\left|
\widetilde{\mathcal{M}}_{\sigma, \mathcal{P}_q(y_1, i-1)}(\mathrm{Sym}^0, x)
\right|
\left|
\widetilde{\mathcal{M}}_{\sigma, p(i)}(\mathrm{Sym}^0, x)
-1
\right|
\nonumber\\
\leq&
\sum_{i=1}^{d}
\left|
\widetilde{\mathcal{M}}_{\sigma, p(i)}(\mathrm{Sym}^0, x)
-1
\right|,
\end{align}
where on the last inequality we used \eqref{M-est-1}.
Here we show
\begin{align}\label{1+error}
\widetilde{\mathcal{M}_{\sigma, p}}(\mathrm{Sym}^0, x)
=1+O\bigg(\frac{|x|+|x|^2}{p^{2\sigma}}\bigg).
\end{align}
In fact, applying \eqref{st_p} (with $\Psi=\psi_x$), we have
\begin{align*}
&\widetilde{\mathcal{M}_{\sigma, p}}(\mathrm{Sym}^0, x)
=\frac{2}{\pi}\int_0^{\pi}
\psi_x\Big(\mathscr{G}_{\sigma, p}(e^{ir\theta})
\Big)\sin^2\theta d\theta\\
&=\frac{2}{\pi}\int_0^{\pi}
(1+ix\mathscr{G}_{\sigma, p}(e^{ir\theta})
+O(|x|^2|\mathscr{G}_{\sigma, p}(e^{ir\theta})|^2)
)\sin^2\theta d\theta,
\end{align*}
and hence, applying \eqref{wp_00}, \eqref{wp_0} and the fact
$$
\frac{2}{\pi}\int_0^{\pi}\sin^2\theta d\theta=1,
$$
we obtain \eqref{1+error}.
From \eqref{1+error} we find that the right-hand side of \eqref{wp_1} is
\begin{align}
\ll&
\sum_{i=1}^{d}\frac{|x|+|x|^2}{p(i)^{2\sigma}}.
\end{align}
Therefore, for any $\varepsilon>0$,
there exists large number $y=y(x,\sigma,\varepsilon)$ such that
\[
\left|
\widetilde{\mathcal{M}}_{\sigma, \mathcal{P}_q(y_2)}(\mathrm{Sym}^0, x)
-
\widetilde{\mathcal{M}}_{\sigma, \mathcal{P}_q(y_1)}(\mathrm{Sym}^0, x)
\right|
<\varepsilon
\qquad
(y_2 >y_1 > y),
\]
for any $\sigma>1/2$.
This implies that there exists the limit
$$
\lim_{y\to\infty} \widetilde{\mathcal{M}}_{\sigma, \mathcal{P}_q(y)}(\mathrm{Sym}^0, x)
=\lim_{y\to\infty} \prod_{p\in \mathcal{P}_q(y)}\widetilde{\mathcal{M}}_{\sigma, p}
(\mathrm{Sym}^0, x)
$$
and its convergence is uniform in $|x|\leq a$ for any real number $a\in\mathbb{R}$.
Moreover, because of \eqref{1+error}, it follows
that the infinite product
$\prod_{\wp\in\mathbb{P}} \widetilde{\mathcal{M}_{\sigma, \wp}}(\mathrm{Sym}^0, x)$ is absolutely convergent
and it is uniform in $|x|\leq a$ for any $\sigma > 1/2$.
\end{proof}

\begin{rem}\label{M-p-conti}
The Fourier inverse transform gives
\begin{align}\label{p-Fourierinverse}
\int_{\mathbb{R}}
\widetilde{\mathcal{M}}_{\sigma, \mathcal{P}_q}(\mathrm{Sym}^0, x) \psi_{-u}(x)
\frac{dx}{\sqrt{2\pi}}
=
\mathcal{M}_{\sigma, \mathcal{P}_q}(\mathrm{Sym}^0, u).
\end{align}
The case $t=1$ of Lemma \ref{lem-abc} (a) implies that 
$\mathcal{M}_{\sigma, \mathcal{P}_q}(\mathrm{Sym}^0, u)$ is continuous
if $\mathcal{P}_q\supset\mathcal{P}_q^0$.
\end{rem}

\begin{rem}\label{M-p-supp}
We have shown in Proposition \ref{M_P} that the support of
$\mathcal{M}_{\sigma, \mathcal{P}_q}(\mathrm{Sym}^0, u)$, which we denote by
${\rm Supp}(\mathcal{M}_{\sigma, \mathcal{P}_q})$, is compact.   Here we show that
when $\mathcal{P}\supset\mathcal{P}_q^0$, then
${\rm Supp}(\mathcal{M}_{\sigma, \mathcal{P}_q})\subset C_{\sigma}(\mathcal{P}_q)$,
where $C_{\sigma}(\mathcal{P}_q)$ denotes the closure of the image of 
$\mathscr{G}_{\sigma, \mathcal{P}_q}$.
In fact, for any $x\notin C_{\sigma}(\mathcal{P}_q)$, we can choose a non-negative
continuous function $\Psi_x$ such that $\Psi_x(x)>0$ and
${\rm Supp}(\Psi_x)\cap C_{\sigma}(\mathcal{P}_q)=\emptyset$.
Then the right-hand side of \eqref{prop-1-1} with $\Psi=\Psi_x$ is equal to 0, and so is
the left-hand side.
However, if $x\in {\rm Supp}(\mathcal{M}_{\sigma, \mathcal{P}_q})$, then the
left-hand side of \eqref{prop-1-1} with $\Psi=\Psi_x$ is positive, because
$\mathcal{M}_{\sigma, \mathcal{P}_q}(\mathrm{Sym}^0, u)$ is non-negative and continuous by
Remark \ref{M-p-conti}.
Therefore $x\notin {\rm Supp}(\mathcal{M}_{\sigma, \mathcal{P}_q})$, and hence
the assertion.
\end{rem}
Lemma \ref{lem-abc} yields the next proposition which is the analogue
of \cite[Proposition~3.4]{im_2011} and \cite[Proposition~3.2]{mu} (whose idea goes back to
\cite{ihara}).     Since the proof is omitted in \cite{im_2011} and \cite{mu}, here we
describe the detailed proof.
%
%
\begin{prop}\label{tildeM}
For any $\sigma >1/2$, the convergence of the limit
\[
\widetilde{\mathcal{M}}_{\sigma}(\mathrm{Sym}^0, x)
=
\lim_{y\to\infty}
\widetilde{\mathcal{M}}_{\sigma, \mathcal{P}_q(y)}(\mathrm{Sym}^0, x)
\]
is uniform in $x\in \mathbb{R}$.
Furthermore,
this convergence is $L^t$-convergence and 
the function $\widetilde{\mathcal{M}}_{\sigma}(\mathrm{Sym}^0, x)$
belongs to $ L^t$ $(1\leq t \leq \infty)$.
\end{prop}
\begin{proof}

By \eqref{lem2-a-pr}, for any $\mathcal{P}_q\supset\mathcal{P}_q^0$, for $t\in [1,\infty)$
we see that
\[
\int_{\mathbb{R}}
\left|
\widetilde{\mathcal{M}}_{\sigma,\mathcal{P}_q}(\mathrm{Sym}^0, x)
\right|^t |dx|
\]
is bounded by a constant depending only on $\sigma,t$, and $\mathcal{P}_q^0$.
Therefore, for any positive real number $\varepsilon>0$,
there exists $R_t=R_t(\varepsilon, \sigma, \mathcal{P}_q^0)>1$ 
such that
\begin{equation}\label{R}
\int_{|x|>R_t}
\left|
\widetilde{\mathcal{M}}_{\sigma,\mathcal{P}_q}(\mathrm{Sym}^0, x)
\right|^t |dx|
<
\frac{\varepsilon}{2^{t+1}}.
\end{equation}
Using \eqref{combined} we can choose $R_{\infty}$ sufficiently large so that
\begin{equation}\label{R1}
\sup_{|x|>R_{\infty}}
\left|\widetilde{\mathcal{M}}_{\sigma,\mathcal{P}_q}(\mathrm{Sym}^0, x)\right|
<
\frac{\varepsilon}{2}
\end{equation}
also holds.
For any real number $y>0$ which satisfies
$\mathbb{P}_q \supset \mathcal{P}_q(y) \supset \mathcal{P}_q^0$,
let $Y$ be a real number satisfying
$\mathbb{P}_q \supset \mathcal{P}_q(Y) \supset \mathcal{P}_q(y)$.
Since \eqref{M-est-1} in Lemma \ref{lemma-M-est} and (c) of Lemma \ref{lem-abc} 
yield that
\begin{align}\label{prop2}
&
\left|
\widetilde{\mathcal{M}}_{\sigma}(\mathrm{Sym}^0, x)
-
\widetilde{\mathcal{M}}_{\sigma, \mathcal{P}_q(y)}(\mathrm{Sym}^0, x)
\right|^t
\nonumber\\
=&
\lim_{Y\to\infty}
\Big|
\prod_{p\in \mathcal{P}_q(Y)}
\widetilde{\mathcal{M}}_{\sigma, p}(\mathrm{Sym}^0, x)
-
\widetilde{\mathcal{M}}_{\sigma, \mathcal{P}_q(y)}(\mathrm{Sym}^0, x)
\Big|^t
\nonumber\\
=&
\lim_{Y\to\infty}
\Big|
\prod_{p\in \mathcal{P}_q(Y)\setminus\mathcal{P}_q(y)}
\widetilde{\mathcal{M}}_{\sigma, p}(\mathrm{Sym}^0, x)
- 1
\Big|^t
\Big|
\widetilde{\mathcal{M}}_{\sigma, \mathcal{P}_q(y)}(\mathrm{Sym}^0, x)
\Big|^t
\nonumber\\
\leq&
\lim_{Y\to\infty}
2^t
\Big|
\widetilde{\mathcal{M}}_{\sigma, \mathcal{P}_q(y)}(\mathrm{Sym}^0, x)
\Big|^t
=
2^t
\Big|
\widetilde{\mathcal{M}}_{\sigma, \mathcal{P}_q(y)}(\mathrm{Sym}^0, x)
\Big|^t,
\end{align}
from \eqref{R} (with $\mathcal{P}_q=\mathcal{P}_q^0$) we see that
\begin{align}\label{p2-proof-1}
&
\int_{|x|>R_t}
\left|
\widetilde{\mathcal{M}}_{\sigma}(\mathrm{Sym}^0, x)
-
\widetilde{\mathcal{M}}_{\sigma, \mathcal{P}_q(y)}(\mathrm{Sym}^0, x)
\right|^t |dx|\notag
\\
\leq&
2^t\int_{|x|>R_t}
\Big|\widetilde{\mathcal{M}}_{\sigma, \mathcal{P}_q(y)}(\mathrm{Sym}^0, x)
\Big|^t|dx|
\leq
\frac{\varepsilon}{2}.
\end{align}
From (c), we see that there exists $y'=y'(\varepsilon,\sigma,t,R_t)
=y'(\varepsilon,\sigma,t,\mathcal{P}_q^0)>0$ such that
\[
\left|
\widetilde{\mathcal{M}}_{\sigma}(\mathrm{Sym}^0, x)
-
\widetilde{\mathcal{M}}_{\sigma, \mathcal{P}_q(y)}(\mathrm{Sym}^0, x)
\right|
<
\bigg(\frac{\varepsilon}{2}\bigg)^{1/t}
\bigg(\frac{1}{2R_t}\bigg)^{1/t}
\]
for $|x|\leq R_t$ and any $y>y'$.
Then we have
\begin{align}\label{p2-proof-2}
&
\int_{|x|\leq R_t}
\left|
\widetilde{\mathcal{M}}_{\sigma}(\mathrm{Sym}^0, x)
-
\widetilde{\mathcal{M}}_{\sigma, \mathcal{P}_q(y)}(\mathrm{Sym}^0, x)
\right|^t|dx|
\leq
\int_{|x|\leq R_t}\frac{\varepsilon}{4R_t}|dx|
=
\frac{\varepsilon}{2}.
\end{align}
From \eqref{p2-proof-1} and \eqref{p2-proof-2} we have
\[
\int_{\mathbb{R}} \left|
\widetilde{\mathcal{M}}_{\sigma}(\mathrm{Sym}^0, x)
-
\widetilde{\mathcal{M}}_{\sigma, \mathcal{P}_q(y)}(\mathrm{Sym}^0, x)
\right|^t
|dx|
<
\varepsilon,
\]
for any $y>y'$,
and so $\widetilde{\mathcal{M}}_{\sigma}(\mathrm{Sym}^0, x)\in L^t$
for $t\in [1, \infty)$.

As for the case $t=\infty$, first notice that
there exists $y''=y''(\varepsilon,\sigma,R_{\infty})
=y''(\varepsilon,\sigma,\mathcal{P}_q^0)$ 
with $\mathbb{P}_q \supset \mathcal{P}_q(y'') \supset \mathcal{P}_q^0$ such that
\[
\sup_{|x|\leq R_{\infty}}
\left|
\widetilde{\mathcal{M}}_{\sigma}(\mathrm{Sym}^0, x)
-
\widetilde{\mathcal{M}}_{\sigma,\mathcal{P}_q(y'')}(\mathrm{Sym}^0, x)
\right|
<
\varepsilon
\]
by (c). From \eqref{R1} and \eqref{prop2} (with $t=1$), we see that
\begin{align*}
\sup_{|x|> R_{\infty}}
\left|
\widetilde{\mathcal{M}}_{\sigma}(\mathrm{Sym}^0, x)
-
\widetilde{\mathcal{M}}_{\sigma,\mathcal{P}_q(y'')}(\mathrm{Sym}^0, x)
\right|
<&
2\sup_{|x|> R_{\infty}}
|\widetilde{\mathcal{M}}_{\sigma, \mathcal{P}_q(y'')}(\mathrm{Sym}^0, x)|
<
\varepsilon,
\end{align*}
hence the case $t=\infty$ follows.
\end{proof}

\begin{rem}\label{unif-in-sigma}
The convergence in the above proposition is also uniform in $\sigma$, in the region
$\sigma\geq 1/2+\varepsilon$ (for any $\varepsilon>0$).
\end{rem}

Let $n\geq 3$.    Recall that
\eqref{combined} holds for any $\mathcal{P}_q(y)\supset\mathcal{P}_q^{0,n}$.
Therefore Proposition \ref{tildeM} implies
\begin{equation}\label{tildeJW}
\widetilde{\mathcal{M}}_{\sigma}(\mathrm{Sym}^0, x)
=O_n\left( (1+|x|)^{-n/2}\right),
\end{equation}
where the implied constant depends on $\sigma$ and $\mathcal{P}_q^{0,n}$.
We also have
\begin{equation}\label{tildeTrivial}
|\widetilde{\mathcal{M}}_{\sigma}(\mathrm{Sym}^0, x)|\leq 1
\end{equation}
from \eqref{trivial_P}.

Finally, we define the function $\mathcal{M}_{\sigma}(\mathrm{Sym}^0, u)$ by
\begin{align}\label{Fourierinverse}
\mathcal{M}_{\sigma}(\mathrm{Sym}^0, u)
=
\int_{\mathbb{R}}\widetilde{\mathcal{M}}_{\sigma}(\mathrm{Sym}^0, x)\psi_{-u}(x)
\frac{dx}{\sqrt{2\pi}}
\end{align}
(analogous to \eqref{p-Fourierinverse}),
where we can see that the right-hand side of this equation is absolutely convergent
since $\widetilde{\mathcal{M}}_{\sigma}(\mathrm{Sym}^0, x)\in L^1$
(or, more quantitatively, by \eqref{tildeJW}).
%
%
\begin{prop}\label{M}
For $\sigma>1/2$,
the function $\mathcal{M}_{\sigma}(\mathrm{Sym}^0, u)$ satisfies the following
five properties.
\begin{itemize}
\item[\rm{(1)}] $\displaystyle \lim_{y\to\infty} \mathcal{M}_{\sigma, \mathcal{P}_q(y)}(\mathrm{Sym}^0, u)= \mathcal{M}_{\sigma}(\mathrm{Sym}^0, u)$
and this convergence is uniform in $u$.
\item[\rm{(2)}] The function $\mathcal{M}_{\sigma}(\mathrm{Sym}^0, u)$ is continuous in $u$ and non-negative.
\item[\rm{(3)}] $\displaystyle \lim_{|u| \to\infty} \mathcal{M}_{\sigma}(\mathrm{Sym}^0, u)=0$.
\item[\rm{(4)}] The functions $\mathcal{M}_{\sigma}(\mathrm{Sym}^0, u)$ and $\widetilde{\mathcal{M}}_{\sigma}(\mathrm{Sym}^0, x)$
are Fourier duals of each other.
\item[\rm{(5)}] $\displaystyle \int_{\mathbb{R}} \mathcal{M}_{\sigma}(\mathrm{Sym}^0, u) \frac{du}{\sqrt{2\pi}}=1$.
\end{itemize}
\end{prop}
This is the analogue of Proposition 3.5 in \cite{im_2011} and
the proof is similar.
\begin{proof}
\par
(1).
From \eqref{p2-proof-1} with $t=1$ we have
\begin{align}\label{p2-proof-1-t=1}
&
\int_{|x|>R_1}
\left|
\widetilde{\mathcal{M}}_{\sigma}(\mathrm{Sym}^0, x)
-
\widetilde{\mathcal{M}}_{\sigma, \mathcal{P}_q(y)}(\mathrm{Sym}^0, x)
\right| |dx|\notag
\\
\leq&
2\int_{|x|>R_1}
\Big|\widetilde{\mathcal{M}}_{\sigma, \mathcal{P}_q(y)}(\mathrm{Sym}^0, x)
\Big||dx|
\leq
\frac{\varepsilon}{2}
\end{align}
for $\mathcal{P}_q^0\subset\mathcal{P}_q(y)$.
On the other hand, from (c), there exists $y'''=y'''(\varepsilon,\sigma,R_1)>0$ such that
\[
\left|
\widetilde{\mathcal{M}}_{\sigma}(\mathrm{Sym}^0, x)
-
\widetilde{\mathcal{M}}_{\sigma, \mathcal{P}_q(y)}(\mathrm{Sym}^0, x)
\right|
<
\frac{\varepsilon}{4R_1}
\]
for any $x$ with $|x|\leq R_1$ and for any $y>y'''$. Then we have
\[
\int_{|x|< R_1}
\left|
\widetilde{\mathcal{M}}_{\sigma}(\mathrm{Sym}^0, x)
-
\widetilde{\mathcal{M}}_{\sigma, \mathcal{P}_q(y)}(\mathrm{Sym}^0, x)
\right|
dx
<
\frac{\varepsilon}{2}.
\]
Therefore for any $y>y'''$ with
$\mathbb{P}_q\supset \mathcal{P}_q(y)\supset\mathcal{P}_q^0$
we have 
\begin{align*}
&
\left|
\mathcal{M}_{\sigma}(\mathrm{Sym}^0, u)
-
\mathcal{M}_{\sigma, \mathcal{P}_q(y')}(\mathrm{Sym}^0, u)
\right|
\\
\leq& 
\int_{\mathbb{R}}
\left|
\widetilde{\mathcal{M}}_{\sigma}(\mathrm{Sym}^0, x)
-
\widetilde{\mathcal{M}}_{\sigma, \mathcal{P}_q(y')}(\mathrm{Sym}^0, x)
\right|
dx
<
\varepsilon.
\end{align*}
\par
(2).
Since $\widetilde{\mathcal{M}}_{\sigma}(\mathrm{Sym}^0, x)\in L^1$
by Proposition~\ref{tildeM},
$\mathcal{M}_{\sigma}(\mathrm{Sym}^0, u)$ is continuous.
Since $\mathcal{M}_{\sigma, \mathcal{P}_q(y)}(\mathrm{Sym}^0, u)$
is non-negative function for any $y > 1$,
$\mathcal{M}_{\sigma}(\mathrm{Sym}^0, u)$ is also non-negative by (1).
\par  
(3). For any  $\varepsilon>0$, there exists $y>0$ such that
\[
\left|
\mathcal{M}_{\sigma}(\mathrm{Sym}^0, u)
-
\mathcal{M}_{\sigma,\mathcal{P}_q(y)}(\mathrm{Sym}^0, u)
\right|
<
\varepsilon
\]
for any $u$ by (1).
Therefore we see that
\[
\left|
\mathcal{M}_{\sigma}(\mathrm{Sym}^0, u)
\right|
<
\left|
\mathcal{M}_{\sigma,\mathcal{P}_q(y)}(\mathrm{Sym}^0, u)
\right|
+
\varepsilon.
\]
Since the support of $\mathcal{M}_{\sigma,\mathcal{P}_q(y)}(\mathrm{Sym}^0, u)$
is compact, we see that
\[
\lim_{|u|\to\infty} \mathcal{M}_{\sigma}(\mathrm{Sym}^0, u)\leq\varepsilon,
\]
and since $\varepsilon$ is arbitrary, the limit should be $0$.
\par
(4). We know $\widetilde{\mathcal{M}}_{\sigma}(\mathrm{Sym}^0, x)$ is in $L^2$.
By the inversion formula, we have
\begin{align}\label{prop-3-pr-4}
\int_{\mathbb{R}} \mathcal{M}_{\sigma}(\mathrm{Sym}^0, u)
\psi_x(u)\frac{du}{\sqrt{2\pi}}
=
\widetilde{\mathcal{M}}_{\sigma}(\mathrm{Sym}^0, x).
\end{align}
\par
(5). 
Putting $x=0$ in \eqref{prop-3-pr-4}, we have
\[
\widetilde{\mathcal{M}}_{\sigma}(\mathrm{Sym}^0, 0)
=
\int_{\mathbb{R}} \mathcal{M}_{\sigma}(\mathrm{Sym}^0, u) \frac{du}{\sqrt{2\pi}}.
\]
On the other hand, from (c) we know
\[
\widetilde{\mathcal{M}}_{\sigma}(\mathrm{Sym}^0, 0)
=
\lim_{y\to\infty}
\prod_{p\in \mathbb{P}_q(y)}
\widetilde{\mathcal{M}}_{\sigma, p}(\mathrm{Sym}^0, 0)
\]
and
\[
\widetilde{\mathcal{M}}_{\sigma, p}(\mathrm{Sym}^0, 0)
=
\int_{\mathbb{R}}\mathcal{M}_{\sigma, p}(\mathrm{Sym}^0, u) \frac{du}{\sqrt{2\pi}}
=1
\]
by Proposition~\ref{M_P}.
\end{proof}
%
%
\section{The key lemma and the deduction of Theorem \ref{main1}}\label{sec4}
\par
For a fixed $\sigma>1/2$
and a finite set $\mathcal{P}_q$
satisfying $\mathcal{P}_q^0\subset\mathcal{P}_q\subset \mathbb{P}_q$,
from \eqref{g_sigma} we see that
\[
\psi_x\circ \mathscr{G}_{\sigma, \mathcal{P}_q}(\alpha_f^r(\mathcal{P}_q))
=
\psi_x(\log L_{\mathcal{P}_q}(\mathrm{Sym}_f^r, \sigma)).
\]
Therefore, to prove our Theorem~\ref{main1},
it is important to consider the average 
\begin{align*}
\Avg(\psi_x\circ \mathscr{G}_{\sigma, \mathcal{P}_q})
=&
\lim_{q^m\to\infty}
\sum_{f\in S_k(q^m)}^{\qquad\prime}
\psi_x\circ \mathscr{G}_{\sigma, \mathcal{P}_q}(\alpha_f^r(\mathcal{P}_q)).
\end{align*}

\begin{rem}\label{rem-P-q-fixed}
We may understand that $\mathcal{P}_q^0$ and $\mathcal{P}_q$ are fixed, when $q^m\to\infty$.
In fact, if all $q$ is bounded by a constant $Q$ and $m\to\infty$, we may assume that
any element of $\mathcal{P}_q^0$ (and of $\mathcal{P}_q$) is larger than $Q$.
If $q\to\infty$, we may assume that $q$ is sufficiently large, so is
larger than any element of $\mathcal{P}_q^0$ and $\mathcal{P}_q$.
\end{rem}
\par
Our first aim in this section is to show the following
%
%
\begin{lem}\label{key}
Let $\mathcal{P}_q$ be a fixed finite subset of $\mathbb{P}_q$. 
In the case $2\leq k < 12$ or $k=14$, for $r=1,2$ (i.e. $\rho=0$), we have
\begin{align*}
\Avg(\psi_x \circ \mathscr{G}_{\sigma, \mathcal{P}_q})
=&
\int_{\Theta_{\mathcal{P}_q}}
\psi_x(
\mathscr{G}_{\sigma, \mathcal{P}_q} (e^{i\theta_{\mathcal{P}_q}r})
)
d^{\rm{ST}} \theta_{\mathcal{P}_q}
\\
=&
\int_{\mathcal{T}_{\mathcal{P}_q}}
\psi_x
\left(
\mathscr{G}_{\sigma, \mathcal{P}_q}(t_{\mathcal{P}_q}^r)
\right)
\prod_{p\in\mathcal{P}_q}
\bigg(\frac{t_p^2 -2 + t_p^{-2}}{-2}\bigg)
d^H t_{\mathcal{P}_q}.
\end{align*}
The above convergence is uniform in $|x|\leq R$ for any $R>0$.
\end{lem}
This lemma is an analogue of \cite[Lemma 4.1]{mu}, and the structure of the proof
is similar.

\begin{proof}
Let $1>\varepsilon'>0$ and $p\in\mathcal{P}_q$.
Considering the Taylor expansion of 
$g_{\sigma, p}(t)=-\log (1-t p^{-\sigma})$, 
we find that there exist an $M_p=M_p(\varepsilon',R)\in\mathbb{N}$ and
$c_{m_p}=c_{m_p}(x,p,\sigma)\in\mathbb{C}$ such that
$\psi_x\circ g_{\sigma, p}$ can be approximated by a polynomial as
\begin{equation}\label{app-uni}
\bigg|
\psi_x \circ g_{\sigma, p}(t)
- 
\sum_{m_p=0}^{M_p} c_{m_p} t^{m_p}
\bigg|
<{\varepsilon'},
\end{equation}
uniformly on $T$ with respect to $t$ and also on $|x|\leq R$ with respect
to $x$. Here $c_0=1$. Write
\[
\phi_{\sigma, p}(t_p ; M_p)=
\sum_{m_p=0}^{M_p} c_{m_p} t_p^{m_p}
\]
and define
\[
\Phi_{\sigma, \mathcal{P}_q}(t_{\mathcal{P}_q};  M_{\mathcal{P}_q})
=
\prod_{p \in \mathcal{P}_q}
\phi_{\sigma, p}(t_p ; M_p)\phi_{\sigma, p}(t_p^{-1} ; M_p)
\phi_{\sigma, p}(\delta_{r,\text{even}} ; M_p),
\]
where $M_{\mathcal{P}_q}=(M_p)_{p\in\mathcal{P}_q}$.
Let $\varepsilon''>0$.
Choosing $\varepsilon'$ (depending
on $|\mathcal{P}_q|$ and $\varepsilon''$) sufficiently small, we obtain
$M_{\mathcal{P}_q}$ such that
\begin{equation}\label{app}
|
\psi_x \circ \mathscr{G}_{\sigma, \mathcal{P}_q}(t_{\mathcal{P}_q})
-
\Phi_{\sigma, \mathcal{P}_q}(t_{\mathcal{P}_q} ; M_{\mathcal{P}_q})| 
< \varepsilon'',
\end{equation}
again uniformly on $T_{\mathcal{P}_q}$
with respect to $t_{\mathcal{P}_q}$ and also on $|x|\leq R$ 
with respect to $x$.
In fact, since
\begin{align*}
&
\psi_x \circ \mathscr{G}_{\sigma, \mathcal{P}_q}(t_{\mathcal{P}_q})
\\
&=
\prod_{p\in\mathcal{P}_q}
\psi_x(g_{\sigma, p}(t_p))\psi_x(g_{\sigma, p}(t_p^{-1}))
\psi_x(g_{\sigma, p}(\delta_{r,\text{even}}))
\\
&=
\prod_{p\in\mathcal{P}_q}
(\phi_{\sigma, p}(t_p ; M_p) +O(\varepsilon'))
(\phi_{\sigma, p}(t_p^{-1} ; M_p)+O(\varepsilon'))
(\phi_{\sigma, p}(\delta_{r,\text{even}} ; M_p)+O(\varepsilon'))
\\
&=\Phi_{\sigma, \mathcal{P}_q}(t_{\mathcal{P}_q} ; M_{\mathcal{P}_q})
+(\text{remainder terms}),
\end{align*}
we obtain \eqref{app}.
Here, in the final stage, we use the fact
$\phi_{\sigma, p}=O(1)$ which follows from \eqref{app-uni}.
%
%
\par
First we 
express the average of the value of
$\psi_x\circ \mathscr{G}_{\sigma, \mathcal{P}_q}$ 
by using $\Phi_{\sigma, \mathcal{P}_q}$.
Let $\varepsilon>0$.
From \eqref{app} with
$t_{\mathcal{P}_q}=\alpha_f(\mathcal{P}_q)$, 
$\varepsilon''=\varepsilon/2$
and \eqref{P1}, we have 
\begin{align}\label{psiPhi-Psi}
&
\left|
\sum_{f\in S_k(q^m)}^{\qquad\prime}
\psi_x \circ \mathscr{G}_{\sigma, \mathcal{P}_q}(\alpha_f^r(\mathcal{P}_q))
-
\sum_{f\in S_k(q^m)}^{\qquad\prime}
\Phi_{\sigma, \mathcal{P}_q}(\alpha_f^r(\mathcal{P}_q) ; M_{\mathcal{P}_q})
\right|
\nonumber\\
< &
\sum_{f\in S_k(q^m)}^{\qquad\prime}
\varepsilon''
=\frac{\varepsilon}{2}(1 +O(E(q^m))).
\end{align}
Therefore, if $q^m$ is sufficiently large, from \eqref{E2}
we see that
\begin{align}\label{step-1}
\left|
\sum_{f \in S(q^m)}^{\qquad\prime} 
\psi_x \circ \mathscr{G}_{\sigma, \mathcal{P}_q}(\alpha_f^r(\mathcal{P}_q))
-
\sum_{f\in S(q^m)}^{\qquad\prime}
\Phi_{\sigma, \mathcal{P}_q}(\alpha_f^r(\mathcal{P}_q); M_{\mathcal{P}_q})
\right|<\varepsilon.
\end{align}
%
%
\par
Next we calculate
$\Phi_{\sigma, \mathcal{P}_q}$ as follows;
\begin{align*}
&
\Phi_{\sigma, \mathcal{P}_q}(\alpha_f^r(\mathcal{P}_q) ;  M_{\mathcal{P}_q})
\\
=&
\prod_{ p\in \mathcal{P}_q}
\Big(\sum_{m_p=0}^{M_p}c_{m_p}e^{ im_pr\theta_f(p)}\Big)
\Big(\sum_{n_p=0}^{M_p}c_{n_p}e^{- in_pr\theta_f(p)}\Big)
\Big(\sum_{\ell_p=0}^{M_p}c_{\ell_p}\delta_{r,\text{even}}^{\ell_p}\Big)
\\
=&
\prod_{p\in \mathcal{P}}
\Big(\sum_{\ell_p=0}^{M_p}c_{\ell_p}\delta_{r,\text{even}}^{\ell_p}\Big)
\Big(
\sum_{m_p=0}^{M_p}c_{m_p}^2
+
\sum_{m_p=0}^{M_p}
\sum_{\substack{n_p=0\\m_p \neq n_p}}^{M_p}
c_{m_p}e^{ im_p r\theta_f(p)}
c_{n_p}e^{- in_p r\theta_f(p)}
\Big)
\\
=&
\prod_{p \in \mathcal{P}}
\Big(\sum_{\ell_p=0}^{M_p}c_{\ell_p}\delta_{r,\text{even}}^{\ell_p}\Big)
\Big(
\sum_{m_p=0}^{M_p}c_{m_p}^2
+
\sum_{m_p=0}^{M_p}
\sum_{\substack{n_p=0 \\ m_p<n_p}}^{M_p}
c_{m_p}c_{n_p}
e^{ i(m_p-n_p) r\theta_f(p)}
\\
&+
\sum_{m_p=0}^{M_p}
\sum_{\substack{n_p=0 \\ m_p<n_p}}^{M_p}
c_{m_p}c_{n_p}
e^{ i(n_p-m_p) r\theta_f(p)}
\Big).
\end{align*}
Putting $n_p-m_p=\nu_p$,
the right-hand side is equal to
\begin{align*}
&
\prod_{ p \in \mathcal{P}_q}
\Big(\sum_{\ell_p=0}^{M_p}c_{\ell_p}\delta_{r,\text{even}}^{\ell_p}\Big)
\Big(
\sum_{m_p=0}^{M_p}c_{m_p}^2
+
\sum_{\nu_p=1}^{M_p}
\sum_{n_p=\nu_p}^{M_p}
c_{n_p-\nu_p}c_{n_p}
\big(
e^{ i\nu_p r\theta_f(p)}
+
e^{- i\nu_p r\theta_f(p)}
\big)
\Big)
\\
=&
\prod_{ p \in \mathcal{P}_q}
\Big(\sum_{\ell_p=0}^{M_p}c_{\ell_p}\delta_{r,\text{even}}^{\ell_p}\Big)
\bigg(
\sum_{m_p=0}^{M_p}c_{m_p}^2
+
\sum_{\nu_p=3}^{M_p}
\sum_{n_p=\nu_p}^{M_p}
c_{n_p-\nu_p}c_{n_p}
\big(
e^{ i\nu_p r\theta_f(p)}
+
e^{- i\nu_p r\theta_f(p)}
\big)
\\
&+
\sum_{n_p=2}^{M_p}
c_{n_p-2}c_{n_p}
\big(
e^{ 2i r\theta_f(p)}
+
e^{- 2i r\theta_f(p)}
\big)
+
\sum_{n_p=1}^{M_p}
c_{n_p-1}c_{n_p}
\big(
e^{ ir\theta_f(p)}
+
e^{- ir\theta_f(p)}
\big)
\bigg).
\end{align*}
Therefore, in the case $r=1$, using \eqref{euler} we see that
\begin{align}\label{odd}
&
\Phi_{\sigma, \mathcal{P}_q}(\alpha_f(\mathcal{P}_q); M_{\mathcal{P}_q})
\nonumber\\
=&
\prod_{p\in \mathcal{P}_q}
\bigg(
\sum_{m_p=0}^{M_p}c_{m_p}^2
+
\sum_{r_p=3}^{M_p}
\sum_{n_p=\nu_p}^{M_p}
c_{n_p-\nu_p}c_{n_p}
\big(
\lambda_f(p^{\nu_p})
-
\lambda_f(p^{\nu_p-2})
\big)
\nonumber\\
&
+
\sum_{n_p=2}^{M_p}
c_{n_p-2}c_{n_p}
\big(
\lambda_f(p^2)-1
\big)
+
\sum_{n_p=1}^{M_p}
c_{n_p-1}c_{n_p}
\lambda_f(p)
\bigg)
\end{align}
In the case $r=2$, similarly we see that
\begin{align}\label{even}
&
\Phi_{\sigma, \mathcal{P}_q}(\alpha_f^2(\mathcal{P}_q)); M_{\mathcal{P}_q})
\nonumber\\
=&
\prod_{ p \in \mathcal{P}_q}
\Big(\sum_{\ell_p=0}^{M_p}c_{\ell_p}\Big)
\bigg(
\sum_{m_p=0}^{M_p}c_{m_p}^2
+
\sum_{\nu_p=3}^{M_p}
\sum_{n_p=\nu_p}^{M_p}
c_{n_p-r_p}c_{n_p}
\big(
\lambda_f(p^{2\nu_p})
-
\lambda_f(p^{2\nu_p-2})
\big)
\nonumber\\
&+
\sum_{n_p=2}^{M_p}
c_{n_p-2}c_{n_p}
\big(
\lambda_f(p^4)-\lambda_f(p^2)
\big)
+
\sum_{n_p=1}^{M_p}
c_{n_p-1}c_{n_p}
\big(
\lambda_f(p^2)-1
\big)
\bigg)
\end{align}
From \eqref{odd} and \eqref{even},
by using \eqref{P} and the multiplicity of $\lambda_f$,  we obtain
\begin{align*}
&
\sum_{f\in S(q^m)}^{\qquad\prime}
\Phi_{\sigma, \mathcal{P}_q}(\alpha_f^r(\mathcal{P}_q) ; M_{\mathcal{P}_q})
\\
=&
\prod_{ p\in \mathcal{P}_q}
\Big(\sum_{\ell_p=0}^{M_p}c_{\ell_p}\delta_{r,\text{even}}^{\ell_p}\Big)
\bigg(
\sum_{m_p=0}^{M_p}c_{m_p}^2
-
\sum_{m_p=0}^{M_p-\diamond}
c_{m_p}c_{m_p+\diamond}
\bigg)\\
&\qquad+ O\Big(\prod_{ p\in \mathcal{P}_q}
\Big(\sum_{\ell_p=0}^{M_p}c_{\ell_p}\delta_{r,\text{even}}^{\ell_p}\Big)
E(q^m)
\Big),
\end{align*}
where $\diamond=2$ if $r=1$ and $\diamond=1$ if $r=2$.
In this estimate,
the implied constant of the error term depends on 
$\mathcal{P}_q$, $x$, $\sigma$ and 
$ M_{\mathcal{P}_q}= M_{\mathcal{P}_q}(\varepsilon',R)$
(hence depends on $\varepsilon$ under the above choice of $\varepsilon'$).
But still, 
this error term is smaller than $\varepsilon$ for sufficiently large $q^m$.
Combining this with \eqref{step-1}, we obtain
\begin{equation}\label{step-2}
\left|
\sum_{f \in S(q^m)}^{\qquad\prime} 
\psi_x \circ \mathscr{G}_{\sigma, \mathcal{P}_q}(\alpha_f^r(\mathcal{P}_q))
-
\prod_{p \in \mathcal{P}_q}
\Big(\sum_{\ell_p=0}^{M_p}c_{\ell_p}\delta_{r,\text{even}}^{\ell_p}\Big)
\bigg(
\sum_{m_p}^{M_p}c_{m_p}^2
-
\sum_{m_p=0}^{M_p-\diamond}
c_{m_p}c_{m_p+\diamond}
\bigg)
\right|
<
2\varepsilon.
\end{equation}
\par
Lastly we calculate the integral
in the statement of Lemma~\ref{key}.
Putting $\Psi=\psi_x$ in \eqref{prop-1-1} of Proposition~\ref{M_P}, we already know 
that
\begin{align*}
& 
\int_{\Theta_{\mathcal{P}_q}}
\psi_x\big(\mathscr{G}_{\sigma, \mathcal{P}_q}(e^{i \theta_{\mathcal{P}_q}r})\big)
d^{\rm ST}\theta_{\mathcal{P}_q}
\\
=&
\int_{\mathcal{T}_{\mathcal{P}_q}}
\psi_x(\mathscr{G}_{\sigma, \mathcal{P}_q}( t_{\mathcal{P}_q}^r))
\prod_{p\in\mathcal{P}_q}
\bigg(\frac{t_p^2-2+t_p^{-2}}{-2}\bigg)
d^H t_{\mathcal{P}_q},
\end{align*}
so our remaining task is to prove that these integrals are equal to
$\Avg(\psi_x\circ \mathscr{G}_{\sigma,\mathcal{P}_q})$.
For any $\varepsilon>0$, using \eqref{app}, we have
\begin{align*}
&
\int_{\mathcal{T}_{\mathcal{P}_q}}
\psi_x(\mathscr{G}_{\sigma, \mathcal{P}_q}(t_{\mathcal{P}_q}^r))
\prod_{p\in\mathcal{P}_q}
\bigg(\frac{t_p^2-2+t_p^{-2}}{-2}\bigg)
d^H t_{\mathcal{P}_q}
\\
=&
\int_{\mathcal{T}_{\mathcal{P}_q}}
\bigg(
\psi_x(\mathscr{G}_{\sigma, \mathcal{P}_q}(t_{\sigma, \mathcal{P}_q}^r))
-
\Phi_{\sigma, \mathcal{P}_q}(t_{\mathcal{P}_q}^r ; M_{\mathcal{P}_q})
\bigg)
\prod_{p\in\mathcal{P}_q}
\bigg(\frac{t_p^2-2+t_p^{-2}}{-2}\bigg)
d^H t_{\mathcal{P}_q}
\\
&+
\int_{\mathcal{T}_{\mathcal{P}_q}}
\Phi_{\sigma, \mathcal{P}_q}(t_{\mathcal{P}_q}^r ; M_{\mathcal{P}_q })
\prod_{p\in\mathcal{P}_q}
\bigg(\frac{t_p^2-2+t_p^{-2}}{-2}\bigg)
d^H t_{\mathcal{P}_q}
\\
=&
\prod_{p\in\mathcal{P}_q}
\int_{\mathcal{T}}
\phi_{\sigma, p}(t_p^r; M_p)\phi_{\sigma, p}(t_p^{-r}; M_p)
\phi_{\sigma, p}(\delta_{r,\text{even}}; M_p)
\bigg(\frac{t_p^2-2+t_p^{-2}}{-2}\bigg)
\frac{dt_p}{2\pi it_p}
+
O(\varepsilon).
\end{align*}
Since
\begin{align*}
&
\int_{\mathcal{T}}
\phi_{\sigma, p}(t_p^r; M_p)\phi_{\sigma, p}(t_p^{-r}; M_p)
\phi_{\sigma, p}(\delta_{r,\text{even}}; M_p)
\bigg(\frac{t_p^2-2+t_p^{-2}}{-2}\bigg)
\frac{dt_p}{2\pi it_p}
\\
=&
\int_{\mathcal{T}}
\bigg(
\sum_{m=0}^{M_p}c_mt_p^{mr}
\bigg)
\bigg(
\sum_{n=0}^{M_p}c_nt_p^{-nr}
\bigg)
\bigg(
\sum_{\ell=0}^{M_p}c_n\delta_{r,\text{even}}^{\ell}
\bigg)
\bigg(\frac{t_p^2-2+t_p^{-2}}{-2}\bigg)
\frac{dt_p}{2\pi it_p}
\\
=&
\bigg(
\sum_{\ell=0}^{M_p}c_n\delta_{r,\text{even}}^{\ell}
\bigg)
\int_0^{2\pi}
\bigg(
\sum_{m=0}^{M_p}c_me^{imr\theta}
\bigg)
\bigg(
\sum_{n=0}^{M_p}c_ne^{-inr\theta}
\bigg)
\bigg(\frac{e^{2i\theta}-2+e^{-2i\theta}}{-2}\bigg)
\frac{d\theta}{2\pi}
\\
=&
\bigg(
\sum_{\ell=0}^{M_p}c_n\delta_{r,\text{even}}^{\ell}
\bigg)
\int_0^{2\pi}
\bigg(
\sum_{m=0}^{M_p}c_m^2
+
\sum_{m=0}^{M_p}
\sum_{\substack{n=0\\ m\neq n}}^{M_p}
c_mc_ne^{i(m -n)r\theta}
\bigg)
\bigg(\frac{e^{2i\theta}-2+e^{-2i\theta}}{-2}\bigg)
\frac{d\theta}{2\pi}
\\
=&
\bigg(
\sum_{\ell=0}^{M_p}c_n\delta_{r,\text{even}}^{\ell}
\bigg)
\begin{cases}
\displaystyle
\sum_{m=0}^{M_p}c_m^2 -\sum_{m=0}^{M_p-2} c_mc_{m+2}
& r=1\\
\displaystyle
\sum_{m=0}^{M_p}c_m^2 -\sum_{m=0}^{M_p-1} c_mc_{m+1}
& r=2,
\end{cases}
\end{align*}
we have
\begin{align}\label{step-3}
&
\int_{\mathcal{T}_{\mathcal{P}_q}}
\psi_x(\mathscr{G}_{\sigma, \mathcal{P}_q}(t_{\mathcal{P}_q}^r)
\prod_{p\in\mathcal{P}_q}
\bigg(\frac{t_p^2-2+t_p^{-2}}{-2}\bigg)
d^H t_{\mathcal{P}_q}
\nonumber\\
=&
\prod_{p\in\mathcal{P}_q}
\bigg(
\sum_{\ell=0}^{M_p}c_n\delta_{r,\text{even}}^{\ell}
\bigg)
\bigg(
\sum_{m_p=0}^{M_p}c_{m_p}^2
-
\sum_{m_p=0}^{M_p-\diamond}
c_{m_p}c_{m_p+\diamond}
\bigg)
+
O(\varepsilon).
\end{align}
From \eqref{step-2} and \eqref{step-3} we find that the identity in the statement 
of Lemma~\ref{key}
holds with the error $O(\varepsilon)$,
but this error can be arbitrarily small if we choose sufficiently large $q^m$,
so the assertion of Lemma~\ref{key} follows.
\end{proof}
%
%

The next lemma is an analogue of \cite[Lemma 4.3]{mu}.

\begin{lem}\label{keylemma}
Suppose Assumptions (Analytical condtions) and (GRH),
in the case  $2\leq k <12$ or $k=14$, for $\sigma>1/2$ and $r=1,2$,
we have
\[
\Avg\psi_x (\log L_{\mathbb{P}_q}(\mathrm{Sym}_f^r, \sigma))
=
\int_{\mathbb{R}}
\mathcal{M}_{\sigma}(\mathrm{Sym}^0, u)
\psi_x(u)\frac{du}{\sqrt{2\pi}},
\]
where $\psi_x(u)=\exp (ixu)$. 
The above convergence is uniform in $|x|\leq R$ for any $R>0$.
\end{lem}
We will prove this lemma in the next section.
We note that this lemma is actually the special case $\Psi=\psi_x$ of our main
Theorem~\ref{main1}.
To show this lemma is the core of the proof of Theorem~\ref{main1}.

\begin{proof}[Proof of Theorem \ref{main1} and Remark \ref{abs-conv-case}]
From Lemma \ref{keylemma},
we can obtain Theorem~\ref{main1} by the same
argument as in \cite[Section~7]{mu} (or in \cite[Section~9]{im_2011}), 
with corrections presented in \cite{matsumoto}.
We omit the details here.
We only note that we use \eqref{tildeTrivial} in the argument.

The assertion in Remark \ref{abs-conv-case} can be deduced easily from
Remark \ref{M-p-supp} as follows. 
First note that, by \eqref{g_sigma}, the support of 
$\log L_{\mathcal{P}_q}(\mathrm{Sym}_f^r,\sigma)$ is included in
$C_{\sigma}(\mathcal{P}_q)$.
If $\sigma>1$, $C_{\sigma}(\mathcal{P}_q)$ remains
bounded when $|\mathcal{P}_q|\to\infty$.   
Let $B_{\sigma}$ be the closure of the union of all 
$C_{\sigma}(\mathcal{P}_q)$, where $\mathcal{P}_q$ runs over all possible
finite subset of $\mathbb{P}_q$.   This is surely a compact set.
By Remark \ref{abs-conv-case} we see that
$\mathrm{Supp}(\mathcal{M}_{\sigma,\mathcal{P}_q})\subset B_{\sigma}$
for any $\mathcal{P}_q$, hence 
$\mathrm{Supp}(\mathcal{M}_{\sigma})\subset B_{\sigma}$.
Therefore $\mathrm{Supp}(\mathcal{M}_{\sigma})$ is compact.

Let $\Psi:\mathbb{R}\to\mathbb{C}$ be a continuous function.   Define
a bounded continuous $\Psi_0$ satisfying $\Psi_0(x)=\Psi(x)$ for 
$x\in B_{\sigma}$ and $\Psi_0(x)=0$ if $|x|$ is sufficiently large.
Then by Theorem \ref{main1}, formula \eqref{main-formula} is valid for
$\Psi=\Psi_0$.   However, since $\Psi_0=\Psi$ on $B_{\sigma}$, we can
replace $\Psi_0$ by $\Psi$ on the both sides of \eqref{main-formula}.
Therefore the condition ``bounded continuous'' in
Theorem \ref{main1} can be relaxed to ``continuous''.
\end{proof}

%
%
\section{Proof of Lemma~\ref{keylemma}}\label{sec5}
We begin with the case $\sigma>1$.

\begin{proof}[Proof of Lemma \ref{keylemma} in the case $\sigma >1$]
Since $\sigma >1$, we can find a sufficiently large finite subset 
$\mathcal{P}_q \subset \mathbb{P}_q$
for which it holds that
\[
|L_{\mathbb{P}_q}(\mathrm{Sym}_f^r, \sigma)
-
L_{\mathcal{P}_q}(\mathrm{Sym}_f^r, \sigma)
|<\varepsilon
\]
and
$|\widetilde{\mathcal{M}}_{\sigma, \mathcal{P}_q}(\mathrm{Sym}^0, x)-\widetilde{\mathcal{M}}_{\sigma}(\mathrm{Sym}^0, x)|<\varepsilon$
for any $x\in\mathbb{R}$ and any $\varepsilon>0$.
The last inequality is provided by Proposition~\ref{tildeM}.
Using this $\mathcal{P}_q$, we have
\begin{align*}
&
\bigg|\sum_{f\in S_k(q^m)}^{\qquad\prime}
\psi_x (\log L_{\mathbb{P}_q}(\mathrm{Sym}_f^r, \sigma))
-
\int_{\mathbb{R}}\mathcal{M}_{\sigma}(\mathrm{Sym}^0, u)
\psi_x(u) \frac{du}{\sqrt{2\pi}}
\bigg|
\\
\leq &
\bigg|\sum_{f\in S_k(q^m)}^{\qquad\prime}
\bigg(\psi_x (\log L_{\mathbb{P}_q}(\mathrm{Sym}_f^r, \sigma))
-
\psi_x (\log L_{\mathcal{P}_q}(\mathrm{Sym}_f^r, \sigma))
\bigg)\bigg|
\\
&+
\bigg|\sum_{f\in S_k(q^m)}^{\qquad\prime}
\psi_x (\log L_{\mathcal{P}_q}(\mathrm{Sym}_f^r, \sigma))
-
\int_{\mathbb{R}}\mathcal{M}_{\sigma, \mathcal{P}_q}(\mathrm{Sym}^0, u)
\psi_x(u) \frac{du}{\sqrt{2\pi}}
\bigg|
\\
&
+
\bigg|
\int_{\mathbb{R}}\mathcal{M}_{\sigma, \mathcal{P}_q}(\mathrm{Sym}^0, u)
\psi_x(u) \frac{du}{\sqrt{2\pi}}
-
\int_{\mathbb{R}}\mathcal{M}_{\sigma}(\mathrm{Sym}^0, u)
\psi_x(u) \frac{du}{\sqrt{2\pi}}
\bigg|
\\
&
=S_1+S_2+S_3,
\end{align*}
say.   We remind the relation 
\begin{equation}\label{ihara}
|\psi_x(u)-\psi_x(u')|
\ll |x|\cdot |u-u'|
\end{equation}
for $u\in \mathbb{R}$ (see Ihara~\cite[(6.5.19)]{ihara} or Ihara-Matsumoto~\cite{im_2011}).
We see that
\begin{align*}
S_1 \ll \; & |x| 
\sum_{f\in S_k(q^m)}^{\qquad\prime}
\bigg(
|\log L_{\mathbb{P}_q}(\mathrm{Sym}_f^r, \sigma)
-
\log L_{\mathcal{P}_q}(\mathrm{Sym}_f^rf, \sigma)|
\bigg)
\end{align*}
and
\[
S_3
=
\left|
\widetilde{\mathcal{M}}_{\sigma,\mathcal{P}_q}(\mathrm{Sym}^0, x)
-
\widetilde{\mathcal{M}}_{\sigma}(\mathrm{Sym}^0, x)
\right|.
\]
Therefore $|S_1|$ and $|S_3|$ are $O(\varepsilon)$
for large $|\mathcal{P}_q|$, with the implied constant depending on $R$.
As for the estimate on $|S_2|$,
we use Proposition~\ref{M_P} and Lemma~\ref{key},
whose convergence is uniform on $|x|\leq R$.
This completes the proof.
\end{proof}

%
\par
Now we proceed to the proof in the more difficult case
$1\geq\sigma>1/2$.
Let
$$
F(\mathrm{Sym}_f^r,s)=\frac{L_{\mathbb{P}_q}(\mathrm{Sym}_f^r, s)}
{L_{\mathcal{P}_q(\log q^m)}(\mathrm{Sym}_f^rf, s)}.
$$
Under Assumption (Analytical conditions), for $2\geq \sigma>1/2$,
we have
\begin{align}\label{10-5-7-anal}
\frac{F^{\prime}}{F}(\mathrm{Sym}_f^r,s)
\ll
\log q^m+(\log(1+|t|))
\end{align}
analogously to \cite[(5.7)]{mu}.
%
In \cite{mu}, we used the assumption $q>Q(\mu)$ (where $\mu$ is a positive integer
and $Q(\mu)$ is the smallest prime satisfying $2^{\mu}/\sqrt{Q(\mu)}<1$), but this type
of assumption is not necessary here because of \eqref{local-factor}. 
From \eqref{10-5-7-anal} we can deduce the following lemma, which is an analogue of
\cite[Lemma~5.1 and Remark~5.2]{mu}, whose basic idea is due to Duke
\cite{d} (Assumption (GRH) is necessary here).
%
%
\begin{lem}\label{appSymL}
Suppose Assumptions (Analytical conditions) and (GRH).
Let $f$ be a primitive form in $S_k(q^m)$,
where $q$ is a prime number.
For fixed
$r\geq 1$ and
$\sigma=1/2 + \delta$ $(0< \delta \leq 1/2)$, 
we have
\begin{align}\label{F}
&
\log L_{\mathbb{P}_q}(\mathrm{Sym}_f^r, \sigma)
- 
\log L_{\mathcal{P}_q(\log q^m)}(\mathrm{Sym}_f^r, \sigma)
-
\mathcal{S}_r
\nonumber\\
\ll_r&
\frac{1}{\delta (\log q^m)^{2\delta}}+(\log q^m)^{-1/2}
+(q^{m/4(k-1)r})^{-\delta/2},
\end{align}
where
\[
\mathcal{S}_r=
\sum_{ p\in \mathbb{P}_q\setminus \mathcal{P}_q(\log q^m)}
\frac{\lambda_f( p^r)}{ p^{\sigma}} e^{- p/q^{m/(k-1)r}}.
\]
\end{lem}
Now let $r=1, 2$ (i.e. $\rho=0$).
Because of \eqref{prop-3-pr-4}, 
the desired assertion is the claim that
\begin{equation}\label{prime}
\bigg|
\sum_{f\in S_k(q^m)}^{\qquad\prime}
\psi_x(\log L_{\mathbb{P}_q}(\mathrm{Sym}_f^r, \sigma))
-
\widetilde{\mathcal{M}}_{\sigma}(\mathrm{Sym}^0, x)
\bigg|
\end{equation}
tends to $0$ as $q^m\to \infty$ for fixed
$q$ or fixed $m$, when $1\geq \sigma > 1/2$.
\par
First, using \eqref{ihara}, we can see the following inequality:
\begin{align}\label{basic}
&
\bigg|
\sum_{f\in S_k(q^m)}^{\qquad\prime}
\psi_x(\log L_{\mathbb{P}_q}(\mathrm{Sym}_f^r, \sigma)) 
-
\widetilde{\mathcal{M}}_{\sigma}(\mathrm{Sym}^0, x)
\bigg|
\nonumber\\
\leq&
\bigg|
\sum_{f\in S_k(q^m)}^{\qquad\prime}
\bigg(
\psi_x(\log L_{\mathbb{P}_q}(\mathrm{Sym}_f^r, \sigma)) 
-
\psi_x(\log L_{\mathcal{P}_q(\log q^m)}(\mathrm{Sym}_f^r, \sigma))
\bigg)
\bigg|
\nonumber\\
&+
\bigg|
\sum_{f\in S_k(q^m)}^{\qquad\prime}
\psi_x(\log L_{\mathcal{P}_q(\log q^m)}(\mathrm{Sym}_f^r, \sigma))
-
\widetilde{\mathcal{M}}_{\sigma, \mathcal{P}_q(\log q^m)}(\mathrm{Sym}^0, x)
\bigg|
\nonumber\\
&+
\left|
\widetilde{\mathcal{M}}_{\sigma, \mathcal{P}_q(\log q^m)}(\mathrm{Sym}^0, x)
-
\widetilde{\mathcal{M}}_{\sigma}(\mathrm{Sym}^0, x)
\right|
\nonumber\\
\ll &
\sum_{f\in S_k(q^m)}^{\qquad\prime}
|x| \bigg(\big|\log L_{\mathbb{P}_q}(\mathrm{Sym}_f^r, \sigma)
- \log L_{\mathcal{P}_q(\log q^m)}(\mathrm{Sym}_f^r, \sigma)
-\mathcal{S}_r\big| +\big|\mathcal{S}_r\big|
\nonumber\\
&+
\bigg|
\sum_{f\in S_k(q^m)}^{\qquad\prime}
\psi_x(\log L_{\mathcal{P}_q(\log q^m)}(\mathrm{Sym}_f^r, \sigma))
-
\widetilde{\mathcal{M}}_{\sigma, \mathcal{P}_q(\log q^m)}(\mathrm{Sym}^0, x)
\bigg|
\nonumber\\
&+
\left|
\widetilde{\mathcal{M}}_{\sigma, \mathcal{P}_q(\log q^m)}(\mathrm{Sym}^0, x)
-
\widetilde{\mathcal{M}}_{\sigma}(\mathrm{Sym}^0, x)
\right|,
\end{align}
this sum being denoted by
\begin{align*}
\mathcal{X}_{\log q^m}+\mathcal{Y}_{\log q^m}
+\mathcal{Z}_{\log q^m},
\end{align*}
say.
From Proposition~\ref{tildeM}, for any $\varepsilon>0$, there exists
a number $N_0=N_0(\varepsilon)$ 
for which
\[
\left|
\widetilde{\mathcal{M}}_{\sigma, \mathcal{P}_q(\log q^m)}(\mathrm{Sym}^0, x)
-
\widetilde{\mathcal{M}}_{\sigma}(\mathrm{Sym}^0, x)
\right|<\varepsilon
\]
holds for any $q^m>N_0$, uniformly in $x\in\mathbb{R}$.
Therefore
\begin{align}\label{lim-Z}
\lim_{q^m\to\infty}\mathcal{Z}_{\log q^m}
=0.
\end{align}
On the estimate of $\mathcal{X}_{\log q^m}$, 
by using \eqref{P1} and \eqref{F},
we find that
\begin{align*}
&
\sum_{f\in S_k(q^m)}^{\qquad\prime}
|x| \bigg(
\big|
\log L_{\mathbb{P}_q}(\mathrm{Sym}_f^r, \sigma)
- \log L_{\mathcal{P}_q(\log q^m)}(\mathrm{Sym}_f^r, \sigma)
-\mathcal{S}_r\big|
\bigg)
\to 0
\end{align*}
as $q^m$ tends to $\infty$, uniformly in $|x|\leq R$.
Next, from \cite[(6.4)]{mu}, we have
\begin{equation}\label{X_{log}}
\sum_{f\in S_k(q^m)}^{\qquad\prime} \big|\mathcal{S}_r\big|
\ll 
(\log q^m)^{-\delta/2}.
\end{equation}
Hence we see that
\begin{align}\label{lim-X}
\lim_{q^m\to\infty} \mathcal{X}_{\log q^m}
=
0
\end{align}
uniformly in $|x|\leq R$.
\par
The remaining part of this section is devoted to the estimate of 
$\mathcal{Y}_{\log q^m}$.
According to the method \cite{im_2011} and \cite{mu},
we begin with the Taylor expansion
\[
\psi_x(g_{\sigma,  p}(t_p))=\exp(ixg_{\sigma,  p}(t_p))
=1+\sum_{n=1}^{\infty}\frac{(ix)^n}{n!}g_{\sigma,  p}^n(t_p),
\]
where
\begin{align*}
g_{\sigma,  p}^n(t_p)
=&
\left(-\log(1-t_p p^{-\sigma})\right)^n
= 
\left(\sum_{j=1}^{\infty}\frac{1}{j}\left(\frac{t_p}{ p^{\sigma}}\right)^j\right)^n
\\
=&
\sum_{a=1}^{\infty} \bigg(\sum_{\substack{a=j_1+\ldots+j_n\\ j_{\ell} \geq 1}}\frac{1}{j_1j_2\cdots j_n}\bigg)
\left(\frac{t_p}{ p^{\sigma}}\right)^a.
\end{align*}
Hence
\begin{align*}
\psi_x(g_{\sigma,  p}(t_p))
= &
1+\sum_{n=1}^{\infty}\frac{(ix)^n}{n!}
\sum_{a=1}^{\infty} \bigg(\sum_{\substack{a=j_1+\ldots+j_n\\ j_{\ell} \geq 1}}\frac{1}{j_1j_2\cdots j_n}\bigg)
\left(\frac{t_p}{ p^{\sigma}}\right)^a\\
=&1+
\sum_{a=1}^{\infty} 
\sum_{n=1}^a\frac{(ix)^n}{n!}\bigg(\sum_{\substack{a=j_1+\ldots+j_n\\ j_{\ell} \geq 1}}\frac{1}{j_1j_2\cdots j_n}\bigg)
\left(\frac{t_p}{ p^{\sigma}}\right)^a,
\end{align*}
which we can write as
\begin{align}\label{G-exp}
\psi_x(g_{\sigma,  p}(t_p))=\sum_{a=0}^{\infty}G_a( p, x) t_p^a
\end{align}
with
\begin{align*}
G_a( p, x)
=& \begin{cases}
       1 & a=0,\\
       \displaystyle{\frac{1}{ p^{a\sigma}}\sum_{n=1}^a\frac{(ix)^n}{n!}\bigg(\sum_{\substack{a=j_1+\ldots+j_n\\ j_{\ell} \geq 1}}\frac{1}{j_1j_2\cdots j_n}\bigg)} 
         & a\geq 1.      
       \end{cases}
\end{align*}
Define
\begin{align*}
G_a(x)
=&
 \begin{cases}
       1 & a=0,\\
       \displaystyle\sum_{n=1}^a\frac{x^n}{n!}
       \left(\begin{matrix}a-1\\n-1\end{matrix}\right)
         & a\geq 1.      
       \end{cases}
\end{align*}
This notation is the same as \cite[(63)]{im_2011}.
We have
\begin{equation}\label{G_a}
|G_a( p, x)|
\leq
\frac{1}{ p^{a\sigma}}G_a(|x|)
\end{equation}(see \cite[(65)]{im_2011}).
From \eqref{g_sigma} and \eqref{G-exp}, we find 
(using the convention $\lambda_f(p^n)=0$ for $n<0$) that
\begin{align}\label{daiji}
&
\sum_{f\in S_k(q^m)}^{\qquad\prime}
\psi_x(\log L_{\mathcal{P}_q(\log q^m)}(\mathrm{Sym}_f^r, \sigma))
=
\sum_{f\in S_k(q^m)}^{\qquad\prime}
\prod_{ p \in \mathcal{P}_q(\log q^m)}
\psi_x (\mathscr{G}_{\sigma,  p}(\alpha_f^r( p)))
\nonumber\\
=&
\sum_{f\in S_k(q^m)}^{\qquad\prime}
\prod_{ p \in \mathcal{P}_q(\log q^m)}
\psi_x (g_{\sigma,  p}(\alpha_f^r(p)))\psi_x(g_{\sigma,  p}(\beta_f^r( p)))
\psi_x(g_{\sigma,  p}(\delta_{r,\text{even}}))
\nonumber\\
=&
\sum_{f\in S_k(q^m)}^{\qquad\prime}
\prod_{ p \in \mathcal{P}_q(\log q^m)}
\Big(\sum_{a_p=0}^{\infty}G_{a_p}( p, x) 
\alpha_f^{a_p r}( p)
\Big)
\Big(\sum_{b_p=0}^{\infty}G_{b_p}( p, x) 
\beta_f^{b_p r}( p)
\Big)
\Big(\sum_{c_p=0}^{\infty}G_{c_p}( p, x) 
\delta_{r,\text{even}}^{c_p}
\Big)
\nonumber\\      
=&
\sum_{f\in S_k(q^m)}^{\qquad\prime}
\prod_{ p \in \mathcal{P}_q(\log q^m)}
\Big(\sum_{c_p=0}^{\infty}G_{c_p}( p, x) 
\delta_{r,\text{even}}^{c_p}
\Big)
\Big(
\sum_{a_p=0}^{\infty} G_{a_p}^2( p, x) 
\nonumber\\
&
+ 
\sum_{0\leq b_p <  a_p}
G_{a_p}( p, x) G_{b_p}( p, x) 
\alpha_f^{a_p r}( p) \beta_f^{b_p r}( p)
\nonumber\\
&
+ 
\sum_{0\leq a_p < b_p} 
G_{a_p}( p, x) G_{b_p}( p, x) 
\alpha_f^{a_p r}( p) \beta_f^{b_p r}( p)
\Big)
\nonumber\\      
=&
\sum_{f\in S_k(q^m)}^{\qquad\prime}
\prod_{ p \in \mathcal{P}_q(\log q^m)}
\Big(\sum_{c_p=0}^{\infty}G_{c_p}( p, x) 
\delta_{r,\text{even}}^{c_p}
\Big)
\Big(
\sum_{a_p=0}^{\infty} G_{a_p}^2( p, x) 
\nonumber\\
&
+ 
\sum_{0\leq b_p}
\sum_{\nu_p=1}^{\infty}
G_{b_p+\nu_p}( p, x) G_{b_p}( p, x) 
\alpha_f^{\nu_p r}( p)
\nonumber\\
&
+ 
\sum_{0\leq a_p}
\sum_{\nu_p=1}^{\infty}
G_{a_p}( p, x) G_{a_p+\nu_p}( p, x) 
\alpha_f^{-\nu_p r}( p) 
\Big)
\nonumber\\      
=&
\sum_{f\in S_k(q^m)}^{\qquad\prime}
\prod_{ p \in \mathcal{P}_q(\log q^m)}
\Big(\sum_{c_p=0}^{\infty}G_{c_p}( p, x) 
\delta_{r,\text{even}}^{c_p}
\Big)
\Big(
\sum_{a_p=0}^{\infty} G_{a_p}^2( p, x)
\nonumber\\
&
+ 
\sum_{0\leq b_p}
\sum_{\nu_p=1}^{\infty}
G_{b_p+\nu_p}( p, x) G_{b_p}( p, x) 
(\alpha_f^{\nu_p r}( p)+ \alpha_f^{-\nu_p r}( p))
\Big)
\nonumber\\
=&
\sum_{f\in S_k(q^m)}^{\qquad\prime}
\prod_{ p \in \mathcal{P}_q(\log q^m)}
\Big(\sum_{c_p=0}^{\infty}G_{c_p}( p, x) 
\delta_{r,\text{even}}^{c_p}
\Big)
\nonumber\\
&
\times\bigg(
\sum_{\nu_p=0}^{\infty}
(\lambda_f(p^{\nu_p r})- \lambda_f(p^{\nu_p r-2}))
\Big(
\sum_{0\leq b_p}G_{b_p+\nu_p}( p, x) G_{b_p}( p, x)
\Big)
\bigg)
\nonumber\\
=:&
\sum_{f\in S_k(q^m)}^{\qquad\prime}
\prod_{ p \in \mathcal{P}_q(\log q^m)}
\Big(\sum_{c_p=0}^{\infty}G_{c_p}( p, x) 
\delta_{r,\text{even}}^{c_p}
\Big)
\nonumber\\
&
\times\bigg(
\sum_{\nu_p=0}^{\infty}
(\lambda_f(p^{\nu_p r})- \lambda_f(p^{\nu_p r-2}))
G_{p, x}(\nu_p)
\bigg),
\end{align}
where
\[
G_{p, x}(\nu_p)
=\sum_{0\leq b_p}G_{b_p+\nu_p}(p,x)G_{b_p}(p,x).
\]
For $r=1$, the right-hand side of \eqref{daiji} is
\begin{align}\label{daiji_r=1}
=&
\sum_{f\in S_k(q^m)}^{\qquad\prime}
\prod_{ p \in \mathcal{P}_q(\log q^m)}
\nonumber\\
&
\bigg(
G_{p, x}(0)
+
\lambda_f(p)
G_{p, x}(1)
+
(\lambda_f(p^{2})- 1)
G_{p, x}(2)
\nonumber\\
&+
\sum_{\nu_p=3}^{\infty}
(\lambda_f(p^{\nu_p})- \lambda_f(p^{\nu_p-2}))
G_{p, x}(\nu_p)
\bigg)
\nonumber\\
=&
\sum_{f\in S_k(q^m)}^{\qquad\prime}
\prod_{ p \in \mathcal{P}_q(\log q^m)}
\nonumber\\
&
\bigg(
(G_{p, x}(0)-G_{p, x}(2))
+
\sum_{\nu_p=1}^{\infty}
\lambda_f(p^{\nu_p})
G_{p, x}(\nu_p)
-
\sum_{\nu_p=1}^{\infty}
\lambda_f(p^{\nu_p})
G_{p, x}(\nu_p+2)
\bigg)
\nonumber\\
=&
\sum_{f\in S_k(q^m)}^{\qquad\prime}
\prod_{ p \in \mathcal{P}_q(\log q^m)}
\bigg(
\sum_{\nu_p=0}^{\infty}
\lambda_f(p^{\nu_p})
\Big(
G_{p, x}(\nu_p)
-
G_{p, x}(\nu_p+2)
\Big)
\bigg).
\end{align}
For $r=2$, the right-hand side of \eqref{daiji} is
\begin{align}\label{daiji_r=2}
=&
\sum_{f\in S_k(q^m)}^{\qquad\prime}
\prod_{ p \in \mathcal{P}_q(\log q^m)}
\Big(\sum_{c_p=0}^{\infty}G_{c_p}( p, x) 
\Big)
\nonumber\\
&
\bigg(
G_{p, x}(0)
+
(\lambda_f(p^{2})- 1)
G_{p, x}(1)
+
\sum_{\nu_p=2}^{\infty}
(\lambda_f(p^{2\nu_p})- \lambda_f(p^{2\nu_p-2}))
G_{p, x}(\nu_p)
\bigg)
\nonumber\\
=&
\sum_{f\in S_k(q^m)}^{\qquad\prime}
\prod_{ p \in \mathcal{P}_q(\log q^m)}
\Big(\sum_{c_p=0}^{\infty}G_{c_p}( p, x) 
\Big)
\nonumber\\
&
\bigg(
\Big(
G_{p, x}(0)-G_{p, x}(1)
\Big)
+
\sum_{\nu_p=1}^{\infty}
\lambda_f(p^{2\nu_p})
G_{p, x}(\nu_p)
-
\sum_{\nu_p=1}^{\infty}
\lambda_f(p^{2\nu_p}))
G_{p, x}(\nu_p+1)
\bigg)
\nonumber\\
=&
\sum_{f\in S_k(q^m)}^{\qquad\prime}
\prod_{ p \in \mathcal{P}_q(\log q^m)}
\Big(\sum_{c_p=0}^{\infty}G_{c_p}( p, x) 
\Big)
\bigg(
\sum_{\nu_p=0}^{\infty}
\lambda_f(p^{2\nu_p})
\Big(
G_{p, x}(\nu_p)
-
G_{p, x}(\nu_p+1)
\Big)
\bigg).
\end{align}
Let
\begin{align*}
  G_p(\text{main}, x)
  =&
  G_{p, x}(0)-G_{p, x}(\diamond)
  =
\begin{cases}
  G_{p, x}(0)-G_{p, x}(2)& r=1\\
  G_{p, x}(0)-G_{p, x}(1)& r=2,
\end{cases}
\\
G_{f,p}(\text{error}, x)
=&
\sum_{\nu_p=1}^{\infty}\lambda_f(p^{\nu_p r})
(G_{p, x}(\nu_p)-G_{p, x}(\nu_p+\diamond))
\\
=&
\begin{cases}
  \displaystyle \sum_{\nu_p=1}^{\infty}\lambda_f(p^{\nu_p})(G_{p, x}(\nu_p)-G_{p, x}(\nu_p+2)) & r=1\\
  \displaystyle \sum_{\nu_p=1}^{\infty}\lambda_f(p^{2\nu_p})(G_{p, x}(\nu_p)-G_{p, x}(\nu_p+1)) & r=2.
\end{cases}
\end{align*}
Write $\mathcal{P}_q(\log q^m)=\{p_1, p_2, \ldots, p_L\}$, where
$p_{\ell}$ means the $\ell$th prime number.
From \eqref{daiji}, \eqref{daiji_r=1}, \eqref{daiji_r=2} and \eqref{P1}
we have
\begin{align}\label{daiji12}
&\sum_{f\in S_k(q^m)}^{\qquad\prime}
  \psi_x(\log L_{\mathcal{P}_q(\log q^m)}(\mathrm{Sym}_f^r, \sigma))
\nonumber\\
=&
\sum_{f\in S_k(q^m)}^{\qquad\prime}
\prod_{ p \in \mathcal{P}_q(\log q^m)}
\Big(\sum_{c_p=0}^{\infty}G_{c_p}( p, x) 
\delta_{r,\text{even}}^{c_p}
\Big)
(G_p(\text{main},x) + G_{f,p}(\text{error},x))
\nonumber\\
=&
\bigg(
\prod_{ p \in \mathcal{P}_q(\log q^m)}
\Big(\sum_{c_p=0}^{\infty}G_{c_p}( p, x) 
\delta_{r,\text{even}}^{c_p}
\Big)
\bigg)
\sum_{f\in S_k(q^m)}^{\qquad\prime}
\bigg(
\prod_{ p \in \mathcal{P}_q(\log q^m)}
G_p(\text{main},x)
\nonumber\\
&
+
\sum_{\substack{(j_0,\ldots,j_L)\neq (0,\ldots,0)\\j_{\ell}\in\{0,1\}}}
\prod_{\ell=1}^L
G_{p_{\ell}}^{1-j_{\ell}}(\text{main},x)
\prod_{\ell=1}^L
G_{f,p_{\ell}}^{j_{\ell}}(\text{error},x)
\bigg)
\nonumber\\
=&
\bigg(
\prod_{ p \in \mathcal{P}_q(\log q^m)}
\Big(\sum_{c_p=0}^{\infty}G_{c_p}( p, x) 
\delta_{r,\text{even}}^{c_p}
\Big)
\bigg)
\nonumber\\
&
\times
\bigg(
\prod_{ p \in \mathcal{P}_q(\log q^m)}
G_p(\text{main},x)
+
E(q^m)
\prod_{ p \in \mathcal{P}_q(\log q^m)}
G_p(\text{main},x)
\nonumber\\
&
+
\sum_{f\in S_k(q^m)}^{\qquad\prime}
\sum_{\substack{(j_0,\ldots,j_L)\neq (0,\ldots,0)\\j_{\ell}\in\{0,1\}}}
\prod_{\ell=1}^L
G_{p_{\ell}}^{1-j_{\ell}}(\text{main},x)
\prod_{\ell=1}^L
G_{f,p_{\ell}}^{j_{\ell}}(\text{error},x)
\bigg).
\end{align}
On the other hand, from \eqref{FFourier-transf},
Proposition~\ref{M_P} and \eqref{G-exp}, 
\begin{align}\label{Daiji}
&
\widetilde{\mathcal{M}}_{\sigma, \mathcal{P}_q(\log q^m)}(\mathrm{Sym}^0, x)
=
\prod_{p \in \mathcal{P}_q(\log q^m)}
\widetilde{\mathcal{M}}_{\sigma, p}(\mathrm{Sym}^0, x)
\nonumber\\
=&
\prod_{p \in \mathcal{P}_q(\log q^m)}
\int_{\mathbb{R}}
\mathcal{M}_{\sigma, p}(\mathrm{Sym}^0, u)
\psi_x(u) du
\nonumber\\
=&
\prod_{p \in \mathcal{P}_q(\log q^m)}
\int_{\mathcal{T}}
\psi_x\big(\mathscr{G}_{\sigma, p}(t_p^r))\big)
\bigg(\frac{t_p^2 -2 + t_p^{-2}}{-2}\bigg)
\frac{dt_p}{2\pi i t_p}
\nonumber\\
=&
\prod_{ p\in \mathcal{P}_q(\log q^m)}
\bigg(
\int_{\mathcal{T}}
\psi_x\big(g_{\sigma, p}( t_p^r)\big)
\psi_x\big(g_{\sigma, p}( t_p^{-r})\big)
\psi_x\big(g_{\sigma, p}(\delta_{r,\text{even}})\big)
\bigg(\frac{t_p^2 -2 + t_p^{-2}}{-2}\bigg)
\frac{dt_p}{2\pi i t_p}
\bigg)
\nonumber\\
=&
\prod_{ p\in \mathcal{P}_q(\log q^m)}
\int_{\mathcal{T}}
\Big(\sum_{a_p=0}^{\infty}G_{a_p}( p, x) t_p^{a_p r} \Big)
\Big(\sum_{b_p=0}^{\infty}G_{b_p}( p, x) t_p^{-b_p r}\Big)
\Big(\sum_{c_p=0}^{\infty}G_{c_p}( p, x)\delta_{r,\text{even}}^{c_p}\Big)
\nonumber\\
&\times
\bigg(\frac{t_p^2 -2 + t_p^{-2}}{-2}\bigg)
\frac{dt_p}{2\pi i t_p}
\bigg)
\nonumber\\
=&
\prod_{ p\in \mathcal{P}_q(\log q^m)}
\Big(\sum_{c_p=0}^{\infty}G_{c_p}( p, x)\delta_{r,\text{even}}^{c_p}\Big)
\bigg(
\int_{\mathcal{T}}
\Big(\sum_{a_p=0}^{\infty}G_{a_p}^2( p, x) 
+
\sum_{a_p=0}^{\infty}
\sum_{\substack{b_p=0\\ a_p\neq b_p}}^{\infty}
G_{a_p}( p, x)G_{b_p}( p, x) t_p^{(a_p-b_p)r}
\Big)
\nonumber\\
&\times
\bigg(\frac{t_p^2 -2 + t_p^{-2}}{-2}\bigg)
\frac{dt_p}{2\pi i t_p}
\bigg)
\nonumber\\
=&
\prod_{ p\in \mathcal{P}_q(\log q^m)}
\Big(\sum_{c_p=0}^{\infty}G_{c_p}( p, x)\delta_{r,\text{even}}^{c_p}\Big)
\bigg(
\int_{\mathcal{T}}
\Big(\sum_{a_p=0}^{\infty}G_{a_p}^2( p, x) 
+
\sum_{a_p=0}^{\infty}
\sum_{\nu_p=1}^{\infty}
G_{a_p}( p, x)G_{a_p+\nu_p}( p, x) t_p^{-\nu_p r}
\nonumber\\
&+
\sum_{b_p=0}^{\infty}
\sum_{\nu_p=1}^{\infty}
G_{b_p+\nu_p}( p, x)G_{b_p}( p, x) t_p^{\nu_p r}
\Big)
\bigg(\frac{t_p^2 -2 + t_p^{-2}}{-2}\bigg)
\frac{dt_p}{2\pi i t_p}
\bigg).
\end{align}
For $r=1$, the right-hand side of \eqref{Daiji} is
\begin{align}\label{Daiji_r=1}
=&
\prod_{ p\in \mathcal{P}_q(\log q^m)}
\int_{\mathcal{T}}
\bigg(
\Big(\sum_{a_p=0}^{\infty}G_{a_p}^2( p, x) 
+
\sum_{a_p=0}^{\infty}
\sum_{\nu_p=1}^{\infty}
G_{a_p}( p, x)G_{a_p+\nu_p}( p, x) t_p^{-\nu_p}
\nonumber\\
&+
\sum_{b_p=0}^{\infty}
\sum_{\nu_p=1}^{\infty}
G_{b_p+\nu_p}( p, x)G_{b_p}( p, x) t_p^{\nu_p}
\Big)
\bigg(\frac{t_p^2 -2 + t_p^{-2}}{-2}\bigg)
\frac{dt_p}{2\pi i t_p}
\bigg)
\nonumber\\
=&
\prod_{ p\in \mathcal{P}_q(\log q^m)}
\bigg(
\sum_{a_p=0}^{\infty}G_{a_p}^2( p, x) 
-
\sum_{a_p=0}^{\infty}
G_{a_p}( p, x)G_{a_p+2}( p, x) 
\bigg)
\nonumber\\
=&
\prod_{ p\in \mathcal{P}_q(\log q^m)}
(G_{p, x}(0)-G_{p, x}(2))
=
\prod_{ p\in \mathcal{P}_q(\log q^m)}
G_p(\text{main},x).
\end{align}
For $r=2$, the right-hand side of \eqref{Daiji} is
\begin{align}\label{Daiji_r=2}
=&
\prod_{ p\in \mathcal{P}_q(\log q^m)}
\Big(\sum_{c_p=0}^{\infty}G_{c_p}( p, x)\Big)
\int_T
\bigg(
\Big(\sum_{a_p=0}^{\infty}G_{a_p}^2( p, x)
\nonumber\\
&+
\sum_{a_p=0}^{\infty}
\sum_{\nu_p=1}^{\infty}
G_{a_p}( p, x)G_{a_p+\nu_p}( p, x) t_p^{-2\nu_p}
\nonumber\\
&+
\sum_{b_p=0}^{\infty}
\sum_{\nu_p=1}^{\infty}
G_{b_p+\nu_p}( p, x)G_{b_p}( p, x) t_p^{2\nu_p}
\Big)
\bigg(\frac{t_p^2 -2 + t_p^{-2}}{-2}\bigg)
\frac{dt_p}{2\pi i t_p}
\bigg)
\nonumber\\
=&
\prod_{ p\in \mathcal{P}_q(\log q^m)}
\Big(\sum_{c_p=0}^{\infty}G_{c_p}( p, x)\Big)
\bigg(
\sum_{a_p=0}^{\infty}G_{a_p}^2( p, x) 
-
\sum_{a_p=0}^{\infty}
G_{a_p}( p, x)G_{a_p+1}( p, x)\bigg)
\nonumber\\
=&
\prod_{ p\in \mathcal{P}_q(\log q^m)}
\Big(\sum_{c_p=0}^{\infty}G_{c_p}( p, x)\Big)
(G_{p, x}(0)-G_{p, x}(1))
\nonumber\\
=&
\prod_{ p\in \mathcal{P}_q(\log q^m)}
\Big(\sum_{c_p=0}^{\infty}G_{c_p}( p, x)\Big)
G_p(\text{main}, x).
\end{align}
From \eqref{Daiji}, \eqref{Daiji_r=1} and \eqref{Daiji_r=2}, we obtain
\begin{equation}\label{Daiji12}
\widetilde{\mathcal{M}}_{\sigma, \mathcal{P}_q(\log q^m)}(\mathrm{Sym}^0, x)
=
\prod_{ p\in \mathcal{P}_q(\log q^m)}
\Big(\sum_{c_p=0}^{\infty}G_{c_p}( p, x)\delta_{r,\text{even}}^{c_p}\Big)
G_p(\text{main}, x).
\end{equation}
Since
\[
\bigg|\sum_{c_p=0}^{\infty} G_{c_p}(p, x)\bigg|
=
|\psi_x(g_{\sigma, p}(1))|
=|e^{ix\log(1-p^{-\sigma})}|=1
\]
by \eqref{G-exp}, from
\eqref{daiji12} and \eqref{Daiji12} we now obtain
\begin{align}\label{Y0}
&
\mathcal{Y}_{\mathcal{P}_q(\log q^m)}
=
\prod_{ p\in \mathcal{P}_q(\log q^m)}
\Big|\sum_{c_p=0}^{\infty}G_{c_p}( p, x)\delta_{r,\text{even}}^{c_p}\Big|
\bigg|
E(q^m)
\prod_{ p \in \mathcal{P}_q(\log q^m)}
G_p(\text{main},x)
\nonumber\\
&
+
\sum_{f\in S_k(q^m)}^{\qquad\prime}
\sum_{\substack{(j_0,\ldots,j_L)\neq (0,\ldots,0)\\j_{\ell}\in\{0,1\}}}
\prod_{\ell=1}^L
G_{p_{\ell}}^{1-j_{\ell}}(\text{main},x)
\prod_{\ell=1}^L
G_{f,p_{\ell}}^{j_{\ell}}(\text{error},x)
\bigg|
\nonumber\\
=&
\bigg|
E(q^m)
\prod_{ p \in \mathcal{P}_q(\log q^m)}
G_p(\text{main},x)
\nonumber\\
&
+
\sum_{\substack{(j_0,\ldots,j_L)\neq (0,\ldots,0)\\j_{\ell}\in\{0,1\}}}
\prod_{\ell=1}^L
G_{p_{\ell}}^{1-j_{\ell}}(\text{main},x)
\bigg(
\sum_{f\in S_k(q^m)}^{\qquad\prime}
\prod_{\ell=1}^L
G_{f,p_{\ell}}^{j_{\ell}}(\text{error},x)
\bigg)
\bigg|
\nonumber\\
\leq &
\bigg|
E(q^m)
\prod_{ p \in \mathcal{P}_q(\log q^m)}
G_p(\text{main},x)
\bigg|
\nonumber\\
&
+
\bigg|
\sum_{\substack{(j_0,\ldots,j_L)\neq (0,\ldots,0)\\j_{\ell}\in\{0,1\}}}
\prod_{\ell=1}^L
G_{p_{\ell}}^{1-j_{\ell}}(\text{main},x)
\bigg(
\sum_{f\in S_k(q^m)}^{\qquad\prime}
\prod_{\ell=1}^L
G_{f,p_{\ell}}^{j_{\ell}}(\text{error},x)
\bigg)
\bigg|
\nonumber\\
=: &
\mathcal{Y}''_{\mathcal{P}_q(\log q^m)}
+
\mathcal{Y}'_{\mathcal{P}_q(\log q^m)},
\end{align}
say.
Here, the definition of 
$\mathcal{Y}''_{\mathcal{P}_q(\log q^m)}$ and $\mathcal{Y}'_{\mathcal{P}_q(\log q^m)}$
are different from those in \cite{mu},
but we will prove that they tend to $0$ by the same argument as in \cite[Section 6]{mu}.
We first consider the inner sum in the definition of
$\mathcal{Y}'_{\mathcal{P}_q(\log q^m)}$.
Let
\[
\mathsf{G}_x(n)
=
\prod_{\substack{1\leq \ell \leq L \\ j_{\ell}=1}}
(G_{p_{\ell}, x}(\nu_{p_{\ell}})-G_{p_{\ell}, x}(\nu_{p_{\ell}}+\diamond)) 
\]
only if $n$ is of the form
\[
n=\prod_{\substack{1\leq \ell \leq L \\ j_{\ell}=1}} p_{\ell}^{\nu_{p_{\ell}} r},
\quad \nu_{p_{\ell}}\geq 1,
\]
and $\mathsf{G}_x(n)=0$ for other $n$.
Let $M$ be a positive number, and
$\eta>0$ be an arbitrarily small positive number.
We see that
\begin{align}\label{Y'0}
  &
  \sum_{f\in S_k(q^m)}^{\qquad\prime}
  \prod_{\ell=1}^L
  G_{f,p_{\ell}}^{j_{\ell}}(\text{error},x)
  \nonumber\\
  =&
  \sum_{f\in S_k(q^m)}^{\qquad\prime}
  \prod_{\ell=1}^L
  \bigg(
  \sum_{\nu_{p_{\ell}}=1}^{\infty}\lambda_f(p_{\ell}^{\nu_{p_{\ell}} r})
  (G_{p_{\ell}, x}(\nu_{p_{\ell}})-G_{p_{\ell}, x}(\nu_{p_{\ell}}+\diamond))
  \bigg)^{j_{\ell}}
  \nonumber\\
  =&
  \sum_{f\in S_k(q^m)}^{\qquad\prime}
  \sum_{n>1}\lambda_f(n)\mathsf{G}_x(n)
  \nonumber\\
  =&
  \sum_{f\in S_k(q^m)}^{\qquad\prime}
  \bigg(
  \sum_{M\geq n>1}\lambda_f(n)\mathsf{G}_x(n)
  +
  \sum_{n> M}\lambda_f(n)\mathsf{G}_x(n)
  \bigg)
  \nonumber\\
\ll&
  \bigg|
\sum_{f\in S_k(q^m)}^{\qquad\prime}
\sum_{M\geq n>1}\lambda_f(n)\mathsf{G}_x(n)
\bigg|
  +
\bigg|  \sum_{f\in S_k(q^m)}^{\qquad\prime}
  \sum_{n>M}\lambda_f(n)\mathsf{G}_x(n)
\bigg|
\nonumber\\
\ll&
  \bigg|
\sum_{f\in S_k(q^m)}^{\qquad\prime}
\sum_{M\geq n>1}\lambda_f(n)\mathsf{G}_x(n)
\bigg|
  +
  \sum_{f\in S_k(q^m)}^{\qquad\prime}
  \sum_{n>M} n^{\eta}|\mathsf{G}_x(n)|
\nonumber\\
\ll&
E(q^m) \sum_{M\geq n>1} n^{(k-1)/2} |\mathsf{G}_x(n)|
  +
\sum_{n>M} n^{\eta}|\mathsf{G}_x(n)|,
\end{align}
where on the last inequality we used \eqref{P}.
Therefore
\begin{align}\label{Y-0-bis}
&\mathcal{Y}'_{\mathcal{P}_q(\log q^m)}\ll
\sum_{\substack{(j_0,\ldots,j_L)\neq (0,\ldots,0)\\j_{\ell}\in\{0,1\}}}
\prod_{\ell=1}^L
G_{p_{\ell}}^{1-j_{\ell}}(\text{main},x)\notag\\
&\times\bigg(E(q^m) \sum_{M\geq n>1} n^{(k-1)/2} |\mathsf{G}_x(n)|
  +
\sum_{n>M} n^{\eta}|\mathsf{G}_x(n)|\bigg).
\end{align}
This corresponds to \cite[(6.13)]{mu}.

By \cite[(6.14) and (6.15)]{mu},
we have the estimates of $G_{p,x}(\nu_p)$ as
\[
|G_{p, x}(\nu_p)|
\leq
\begin{cases}
  \displaystyle \bigg(\exp\bigg(\frac{|x|}{p^{\sigma}-1}\bigg)\bigg)^2 & \nu_p=0\\
  \displaystyle \frac{1}{p^{\nu_p \sigma/2}}\bigg(\exp\bigg(\frac{|x|}{p^{\sigma}-1}\bigg)\bigg)^2\bigg(\exp\bigg(\frac{|x|}{p^{\sigma/2}-1}\bigg)\bigg)
  & \nu_p \neq 0.
  \end{cases}
\]
These estimates yield that for $\sigma>1/2$ and $|x|<R$, there exists a large
$p_0=p_0(\sigma,R)$ for which 
\begin{align*}
  &
  |G(\text{main}, x)|
  =
  |G_{p, x}(0)-G_{p, x}(\diamond)|
  \nonumber\\
  \leq&
  \bigg(\exp\bigg(\frac{|x|}{p^{\sigma}-1}\bigg)\bigg)^2\bigg(1+
  \frac{1}{p^{\nu_p\sigma/2}}\exp\bigg(\frac{|x|}{p^{\sigma/2}-1}\bigg)\bigg)
  \leq
  2
\end{align*}
holds for any $p>p_0$.   Therefore
\begin{align}\label{Y''}
\mathcal{Y}''_{\mathcal{P}_q(\log q^m)}
=
|E(q^m)|
\prod_{\ell=1}^L
|G_{p_{\ell}}(\text{main},x)|
\ll_{\sigma,R} 
|E(q^m)| 2^L \to 0
\end{align}
as $q^m\to \infty$ by the same argument as in \cite[p.~2077]{mu}.
Further, 
since $\nu_p\geq 1$ we have
\begin{align}\label{Y'1}
|\mathsf{G}_x(n)|
\leq &
\prod_{\substack{1\leq \ell \leq L\\j_{\ell}=1}}
(2\max\{|G_{p_{\ell}, x}(\nu_p)|,\; |G_{p_{\ell}, x}(\nu_p+\diamond)|\})
\nonumber\\
\leq&
\frac{2^L}{n^{\sigma r/2}}
\prod_{\substack{1\leq \ell \leq L\\j_{\ell}=1}}
\bigg(\exp\bigg(\frac{|x|}{p^{\sigma}-1}\bigg)\bigg)^2\bigg(\exp\bigg(\frac{|x|}{p^{\sigma/2}-1}\bigg)\bigg).
\end{align}
From the estimates \eqref{Y-0-bis} and \eqref{Y'1},
we see that
$\mathcal{Y}'_{\log q^m}\to 0$ as $q^m\to\infty$
by the same argument as in \cite[pp.~2075--2077]{mu}.
Therefore we conclude that
\begin{equation}\label{lim-Y}
 \lim_{q^m\to\infty} \mathcal{Y}_{\log q^m}\to 0.
\end{equation}
\par
Finally we see that Lemma~\ref{keylemma} is established,
by substituting \eqref{lim-Z}, \eqref{lim-X} and \eqref{lim-Y}
into \eqref{basic}.
\par
%
%
\section{Proof of Corollary~\ref{main2} for $r\geq 3$.}\label{sec6}
\par
In the previous sections,
we proved the formula of Corollary~\ref{main2} for $r=1$ and $r=2$.
We suppose Corollary~\ref{main2} is valid for all $r < \mathfrak{r}$.
From \cite[Theorem~1.5]{mu} we know that
there exists $\mathcal{M}_{\sigma}^*$ such that
\[
  \int_{\mathbb{R}}\mathcal{M}_{\sigma}^*(u)\Psi(u) \frac{du}{\sqrt{2\pi}}
  =
  \mathrm{Avg}\Psi(\log L_{\mathbb{P}_q}(\mathrm{Sym}_f^{\mathfrak{r}}, \sigma)
  -
  \log L_{\mathbb{P}_q}(\mathrm{Sym}_f^{\mathfrak{r}-2}, \sigma)).
  \]
Now assume $\Psi(x+y)=\Psi(x) + \Psi(y)$.
Then
\begin{align}\label{cor-pr-1}
  &
  \int_{\mathbb{R}}\mathcal{M}_{\sigma}^*(u)\Psi(u) \frac{du}{\sqrt{2\pi}}
  \notag\\
  =&
  \mathrm{Avg}\Psi(\log L_{\mathbb{P}_q}(\mathrm{Sym}_f^{\mathfrak{r}}, \sigma))
  -
  \mathrm{Avg}\Psi(\log L_{\mathbb{P}_q}(\mathrm{Sym}_f^{\mathfrak{r}-2}, \sigma)),
\end{align}
and by induction assumption
\begin{align}\label{cor-pr-2}
\mathrm{Avg}_{\star}\Psi
(
\log L_{\mathbb{P}_q}(\mathrm{Sym}_f^{\mathfrak{r}-2}, \sigma)
)
=&
\int_{\mathbb{R}}
\mathcal{M}_{\sigma}(\mathrm{Sym}^0, u)
\Psi(u) du
\notag\\
&+
\rho'
\int_{\mathbb{R}}
\mathcal{M}_{\sigma}^*(u)
\Psi(u) du,
\end{align}
where $\rho'=(\mathfrak{r}-3)/2$ if $\mathfrak{r}-2$ is odd and
$\rho'=(\mathfrak{r}-2)/2 -1$ if $\mathfrak{r}-2$ is even.
Since
\[
\rho'+1
=
\begin{cases}
  (\mathfrak{r}-1)/2 & \mathfrak{r}-2 \;\text{is odd}\\
  \mathfrak{r}/2 -1 & \mathfrak{r}-2 \;\text{is even}
\end{cases}
=
\rho,
\]
from \eqref{cor-pr-1} and \eqref{cor-pr-2} we obtain
\[
\mathrm{Avg}_{\star}
\Psi(\log L_{\mathbb{P}_q}(\mathrm{sym}_f^{\mathfrak{r}}, \sigma)
=
\int_{\mathbb{R}}
\mathcal{M}_{\sigma}(\mathrm{Sym}^0, u)
\Psi(u) du
+\rho
\bigg(
\int_{\mathbb{R}}
  \mathcal{M}_{\sigma}^*(u)
  \Psi(u) du
  \bigg).
\]
If $\Psi$ is continuous and satisfies $\Psi(x+y)=\Psi(x) + \Psi(y)$, then
$\Psi(x)=cx$ with a certain constant $c$.   Therefore we obtain Corollary~\ref{main2}.
%
%

%
\noindent
Philippe Lebacque:\\
Laboratories de Math\'{e}matiques de Besan\c{c}on,\\
UFR Sciences dt techniques 16, route de Gray 25 030 Besan\c{c}on, France.\\
philippe.lebacque@univ-fcomte.fr
\bigskip
\\
\noindent
Kohji Matsumoto:\\
Graduate School of Mathematics,
Nagoya University, Furocho, Chikusa-ku, Nagoya 464-8602, Japan.\\
kohjimat@math.nagoya-u.ac.jp
\bigskip
\\
\noindent
Yumiko Umegaki:\\
Department of Mathematical and Physical Sciences,
Nara Women's University,
Kitauoya Nishimachi, Nara 630-8506, Japan.\\
ichihara@cc.nara-wu.ac.jp
\end{document}